\documentclass{tsp-short}
\usepackage{amsfonts}
\usepackage{amssymb}
\usepackage{newlfont}
\usepackage{amsthm}
\usepackage{euscript}
\usepackage{amsmath}
\usepackage{graphicx}
\usepackage{dsfont}
\usepackage{enumerate}
\usepackage{color}

\newcommand{\Var}{\mathop \mathrm{Var}}
\newcommand{\llim}{\mathop \mathrm{l.i.m.}}

\newcommand{\supp}{\mathop \mathrm{supp}}

\newcommand{\prt}{\partial}

\newtheorem{Proposition}{Proposition}
\theoremstyle{definition}
\newtheorem{Example}{Example}
\newtheorem{Definition}{Definition}
\theoremstyle{plain}
\newtheorem{Lemma}{Lemma}
\theoremstyle{plain}
\newtheorem{Theorem}{Theorem}
\theoremstyle{condition}

\theoremstyle{plain}

\theoremstyle{remark}
\newtheorem{Remark}{Remark}

\newcommand{\vp}{\varphi}
\newcommand{\vf}{\varphi}

\renewcommand{\r}{\rho}

\newcommand{\R}{\mathds{R}^d}
\renewcommand{\P}{\mathds{P}}
\newcommand{\E}{\mathds{E}}

\begin{document}

\title[]{A representation for the derivative with respect to the initial data of the solution of an SDE with a non-regular drift and a Gaussian noise}

\author{Olga V. Aryasova}
\address{Institute of Geophysics, National Academy of Sciences of Ukraine,
Palladin pr. 32, 03680, Kiev-142, Ukraine}
\email{oaryasova@gmail.com}
\author{Andrey Yu. Pilipenko}
\address{Institute of Mathematics,  National Academy of Sciences of
Ukraine, Tereshchenkivska str. 3, 01601, Kiev, Ukraine; National Technical University of Ukraine "KPI", Kiev, Ukraine}
\email{pilipenko.ay@yandex.ua}

 \dedicatory{}


\begin{abstract}
We consider a multidimensional SDE with a Gaussian noise and a drift vector being  a vector function of bounded variation. We prove the existence of generalized derivative of the solution with respect to the initial conditions and represent the derivative as a solution of a linear SDE with coefficients depending on the initial process. The  representation obtained is a natural generalization of the expression for the derivative in the smooth case. The theory of continuous additive functionals is used.
\end{abstract}

\maketitle \thispagestyle{empty}
\section*{Introduction}
Consider a $d$-dimensional nonhomogeneous stochastic differential equation (SDE)
\begin{equation}\label{eq_main}
\left\{
\begin{aligned}
d\vp_t(x)&=a(t,\vp_t(x))dt+\sum_{k=1}^{m}\sigma_{k}(t,\vp_t(x))dw_{k}(t),\\
\vp_0(x)&=x,\\
\end{aligned}\right.
\end{equation}
where $x\in\mathbb{R}^d,  d\geq1, m\geq1$,  $(w(t))_{t\geq0}=(w_{1}(t),\dots,w_{m}(t))_{t\geq0}$ is a standard  $m$-dimensional Wiener process, the drift coefficient $a:[0,\infty)\times\R\to\R$ is Borel measurable and bounded, and the diffusion coefficient $\sigma:[0,\infty)\times\R\to\R\times\mathds{R}^{m}$ is bounded and continuous.

In what follows we suppose that $\sigma$ satisfies the following conditions:
\begin{enumerate}[(C1)]
\item $\sigma\in W_{2d+2,loc}^{0,1}([0,\infty)\times\R)$.
\item {\it Uniform ellipticity}: For each $T>0$, there exists an ellipticity constant $B>0$ such that for all $t\in[0,T]$, $x\in\R$, $\theta\in\R$,
$$
\theta^\ast\sigma(t,x)\sigma^\ast(t,x)\theta\geq B|\theta|^2,
$$
where $|\cdot|$ is a norm in $\R$.
\end{enumerate}

Under these assumptions on the coefficients there exists a unique strong solution to equation (\ref{eq_main}) (see \cite{Veretennikov81}).

It is well known (cf. \cite{Kunita90}) that if the coefficients of
(\ref{eq_main}) are continuously differentiable and the derivatives are bounded and H{\"{o}}lder
continuous uniformly in $t$, then there exists a flow of diffeomorphisms for equation
(\ref{eq_main}). The derivative $\nabla\varphi_t(x)=:Y_t(x)$  is a solution of
the equation
\begin{equation}\label{eq_deriv_new}
dY_t(x)=\nabla a(t,\vp_t(x))Y_t(x)dt+\sum_{k=1}^m\nabla\sigma_k(t,\vp_t(x))Y_t(x) dw_k(t),
\end{equation}
where for a function $f:\R\to\R$, we set $\nabla f=\left(\frac{\partial f^i}{\partial x_j}\right)_{1\leq i,j\leq d}$.

Flandoli et al. \cite{Flandoli+10} showed that  the conditions on the coefficients can be essentially weakened and a flow of diffeomorphisms exists in the case of a
smooth, bounded, uniformly non-degenerate noise and a bounded, uniformly in time H\"{o}lder
continuous drift term.

The case of discontinuous drift was studied in \cite{Fedrizzi+13b, Fedrizzi+13a, MeyerBrandis+13, Mohammed+12} and the weak differentiability of the solution to (\ref{eq_main}) was proved  under rather weak assumptions on the drift. Fedrizzi et al. \cite{Fedrizzi+13b} considered equation \eqref{eq_main} with an identity diffusion matrix and a drift vector belonging to $L_q(0,T; L_p(\mathds{R}^d))$ for some $p,q $ such that
$$
p\geq2, \ q>2, \ \frac dp+ \frac 2q<1.
$$
Using a Zvonkin-type transformation they established the existence of the G\^{a}teaux derivative with respect to the initial data in $L_2(\Omega\times[0,T];\R)$.
The authors of \cite{Mohammed+12} based on the Malliavin calculus proved that the solution of equation \eqref{eq_main} with a bounded measurable drift vector $a$ and an identity diffusion matrix belongs to the space $L^2(\Omega; W^{1,p}(U))$ for each $t\in\R, p>1,$ and any open and bounded $U\in\R$. The Malliavin calculus is used also in \cite{MeyerBrandis+13}. Unfortunately, in these works no representations for the derivatives are given.

The one-dimensional case was considered in \cite{Aryasova+12, Attanasio10}  and explicit expressions for the Sobolev derivative were obtained.
The formulas involve the local time of the initial process.
There are no direct generalizations of these formulas to the multidimensional case because the local time at a point does not exist  in the multidimensional situation.

The aim of the present paper is to get a natural representation for the derivative $\nabla_x\vp_t(x)$ of the solution to equation \eqref{eq_main}.
 We assume that $\sigma$
satisfies (C1),(C2), the H\"older condition, and for some $\rho>0$ and all $1\leq k\leq m$, $1\leq i,j\leq d$, the function $\left|\frac{\partial \sigma_k^i}{\partial y_j}(s,y)\right|^{2+\r}$ belongs to the Kato-type class $\mathcal{K}$, i. e.,
\begin{multline*}
\lim_{t\downarrow 0}\sup_{t_0\in[0,\infty), \ x_0\in\R}\int_{t_0}^{t_0+t}ds\int_{\R}\frac{1}{(2\pi(s-t_0))^{d/2}}\exp\left\{-\frac{|y-x_0|^2}{2(s-t_0)}\right\}\times \\
\left|\frac{\partial \sigma_k^i}{\partial y_j}(s,y)\right|^{2+\r}dy=0.
\end{multline*}
We show that the derivative $Y_t(x)=\nabla_x\varphi_t(x)$ is a
solution to the SDE
\begin{equation}\label{eq_derivative_1}
Y_t(x)=E+\int_0^t dA_s(\varphi(x))Y_s(x)+\sum_{k=1}^{m}\int_0^t\nabla\sigma_{k}(s,\varphi_s(x))Y_s(x)dw_{k}(s),
\end{equation}
where $E$ is the $d$-dimensional identity matrix, $A_s(\varphi(x))$ is a continuous additive functional of the process $(t,\vp_t(x))_{t\geq0}$, which is equal to $\int_0^t \nabla a(s,\vp_s(x))ds$ if $a$ is differentiable. This representation is a natural generalization of the expressions for the smooth case.

We prove the main result for such $a$ that for each $t\geq 0$ and all $1\leq i \leq d$,  $a^i(t,\cdot)$ is a function of bounded
variation on $\R$, i.e., for each $1\leq j\leq d, $ the generalized
derivative $\mu^{ij}(t,dy)=\frac{\partial a^i}{\partial x_j}(t,dy)$ is a
signed measure on $\R$. Besides, we suppose that for all
$1\leq i,j \leq d$, $\mu^{ij}(t,dy)dt$ is of the class $\mathcal{K}$, i.e.,
$$
\lim_{t\downarrow 0}\sup_{t_0\in[0,\infty), \ x_0\in\R}\int_{t_0}^{t_0+t}ds\int_{\R}\frac{1}{(2\pi(s-t_0))^{d/2}}\exp\left\{-\frac{|y-x_0|^2}{2(s-t_0)}\right\}|\mu|^{ij}(s,dy)=0,
$$
where $|\mu|^{ij}=\mu^{ij,+}+\mu^{ij,-}$ is the variation of $\mu^{ij}$; $\mu^{ij,+},\mu^{ij,-}$ are  measures from the Hahn-Jordan
decomposition $\mu^{ij}=\mu^{ij,+}-\mu^{ij,-}$.

The similar results for a homogeneous SDE with an identity diffusion matrix and a drift being a vector function of bounded variation were obtained in \cite{Aryasova+14a}. In this case there is no martingale member in the right-hand side of \eqref{eq_derivative_1}. This essentially simplifies the proof. The argument is based on the theory of additive functionals of homogeneous Markov processes developed by Dynkin \cite{Dynkin63}. In  \cite{Bogachev+15} the same method was applied to a homogeneous SDE with L\'evi noise and a drift being a vector function of bounded variation. The existence of a strong solution and the differentiability of the solution with respect to the initial data were proved. Unfortunately, the theory by Dynkin can not be directly applied to our problem because $(\varphi_t(x))_{t\geq0}$ is not homogeneous.

 The paper is organized as follows. In Section \ref{Section_Preliminaries} we collect some facts from the theory of additive functionals of homogeneous Markov processes  by Dynkin \cite{Dynkin63}.   We intend  to consider a homogeneous process $(t,\varphi_t)_{t\geq0}$  and adapt Dynkin's theory to the functionals of this process. The main result is formulated in Section \ref{Section_main_result} and proved in Section \ref{Proof of Theorem_main}. The idea of the proof is to approximate the solution of equation \eqref{eq_main} by solutions of equations with smooth coefficients. The key point is the convergence of continuous homogeneous additive functionals of the approximating processes  to a functional of the process being the solution to  \eqref{eq_main} (Lemma \ref{Lemma_converg_W_functionals}). The proof of the corresponding statement uses essentially the result on the convergence of the transition probability densities of the approximating processes, which is obtained in Section \ref{Section_Conv_Densities}.

The method proposed can be considered as a generalization of the local time approach used in the one dimensional case.

\section{Preliminaries: continuous additive functionals}\label{Section_Preliminaries}
Let $(\xi_t, \mathcal{F}_t, P_z)$ be a c\'adl\'ag homogeneous Markov process  with a phase space $(E,\mathcal{B})$, where $\sigma$-algebra $\mathcal{B}$ contains all one-point sets  (see notations in \cite{Dynkin63}). Assume that $(\xi_t)_{t\geq0}$ has the infinite life-time. Denote $\mathcal N_t=\sigma\left\{\xi_s:0\leq s\leq t\right\}$
\begin{Definition}\label{def_W_func}
A  random function $A_t, t\geq 0,$ adapted to the filtration $\{\mathcal{N}_t\}$   is called a non-negative continuous additive functional of the process $(\xi_t)_{t\geq0}$ if it is
\begin{itemize}
\item non-negative;
\item continuous in $t$;
\item homogeneous additive, i.e., for all $t\geq 0, \ s>0,$ $z\in E,$
\begin{equation}\label{eq_additive}
A_{t+s}=A_s+\theta_sA_t \ \ P_z-\mbox{almost surely},
\end{equation}
where $\theta$ is the shift operator.
\end{itemize}
If additionally for each $t\geq 0,$
$$
\sup_{z\in E}\mathds{E}_zA_t<\infty,
$$
then $A_t, t\geq 0,$ is called  a W-functional.
\end{Definition}

\begin{Remark}
It follows from Definition \ref{def_W_func} that a W-functional is non-decreasing in $t$, and for all $z\in E$
$$
P_z\{A_0=0\}=1.
$$
\end{Remark}
\begin{Definition}
The function
$$
f_t(z)=\mathds{E}_z A_t
$$
is called the characteristic of a $W$-functional $A_t.$
\end{Definition}
\begin{Remark}\label{Remark_triangle_ineq}(See \cite{Dynkin63}, Properties 6.15). For all $s\geq0$, $t\geq0$,
$$
\|f_{t+s}\|_E\leq \|f_{t}\|_E+\|f_{s}\|_E,
$$
where $\|f_t\|_E=\sup_{z\in E}|f_{t}(z)|$.
\end{Remark}
\begin{Proposition}[See \cite{Dynkin63}, Theorem 6.3]\label{Proposition_uniquely_defined}
A W-functional is defined by its characteristic uniquely up to equivalence.
\end{Proposition}
The following theorem states the relation between the convergence of
W-functionals and the convergence of their characteristics.
\begin{Theorem}[See \cite{Dynkin63}, Theorem 6.4]\label{Theorem_Convergence_characteristics} Let $A_{n,t}, \
n\geq 1,$ be W-functionals of the process $(\xi_t)_{t\geq0}$  and $f_{n,t}(z)=\mathds{E}_z A_{n,t}$ be
their characteristics. Suppose that for each $t>0$, a function
$f_t(z)$ satisfies the condition
\begin{equation}\label{eq_uni_conv_char}
\lim_{n\to\infty}\sup_{0\leq u\leq t}\sup_{z\in E}|f_{n,u}(z)-f_u(z)|=0.
\end{equation}
Then $f_t(z)$ is the characteristic of a W-functional $A_t$. Moreover,
\begin{equation*}
A_t=\llim_{n\to\infty} A_{n,t},
\end{equation*}
where $\llim$ denotes the convergence in mean square (for any initial distribution $\xi_0$).
\end{Theorem}
\begin{Proposition}[See \cite{Dynkin63}, Lemma 6.1$^\prime$]\label{Proposition_uniform_convergence} If for any $t\geq 0$ the
sequence of non-negative additive functionals
$\left\{A_{n,t}: n\geq 1\right\}$ of the Markov process
$(\xi_t)_{t\geq 0}$ converges in probability to a continuous
functional $A_t$, then the convergence in probability is
uniform, i.e.
$$
\forall \ T>0 \sup_{t\in[0,T]}|A_{n,t}-A_t|\to 0, \
n\to\infty, \ \mbox{in probability}.
$$
\end{Proposition}
\begin{Example}\label{Example_general} Let $E=\R$,
$h$ be a non-negative bounded measurable function on $E$, let the process $(\xi_t)_{t\geq0}$  has a transition probability density $g_t(z_1,z_2)$. Then
$$
A_t:=\int_0^t h(\xi_s)ds
$$
is a $W$-functional of the process $(\xi_t)_{t\geq0}$ and its characteristic is equal to
$$
f_t(z)=\int_{E}\left(\int_0^tg_s(z,v)ds\right)h(v)dv=\int_{E}k_t(z,v)h(v)dv,
$$
where
$$
k_t(z,v)=\int_0^tg_s(z,v)ds.
$$

Let a measure $\nu$ be such that $\int_{E}k_t(z,v)\nu(dv)$ is well defined. If we can choose a sequence of non-negative bounded continuous functions $\{h_n: n\geq1\}$ such that for each $T>0,$
$$
\lim_{n\to\infty}\sup_{t\in[0,T]}\sup_{z\in E}\left|\int_{E}k_t(z,v)h_n(v)dv-\int_{E}k_t(z,v)\nu(dv)\right|=0,
$$
then by Theorem \ref{Theorem_Convergence_characteristics} there exists a W-functional $A_t^\nu$ corresponding to the measure $\nu$ with its characteristic being equal to $\int_{E}k_t(z,v)\nu(dv)$.
\end{Example}

Given a measure $\nu$, a sufficient condition for the existence of a corresponding W-functional is as follows.
\begin{Theorem}[See \cite{Dynkin63}, Theorem 6.6]\label{Theorem_sufficient_condition}  Let the condition
\begin{equation}
\lim_{t\downarrow 0}\sup_{z\in E}f_t(z)=\lim_{t\downarrow 0}\sup_{z\in E}\int_{E}k_t(z,y)\nu(dy)=0
\end{equation}
hold. Then $f_t(z)$
is the characteristic of a W-functional
$A_t^{\nu}$. Moreover,
$$
A_t^{\nu}=\llim_{\varepsilon\downarrow 0}\int_0^t \frac{f_\varepsilon(\xi_u)}{\varepsilon}du,
$$
and the sequence of characteristics of integral functionals $\int_0^t \frac{f_\varepsilon(\xi_u)}{\varepsilon}du$ converges to $f_t(z)$ in sense of the relation \eqref{eq_uni_conv_char}.
\end{Theorem}


Cosider a process $\eta_t=(\eta^1_t,\eta^2_t), t\geq 0,$ which is a (unique) solution to the system of SDEs:
\begin{equation}\label{eq_eta}
\left\{\begin{aligned}
d\eta^1_t&=dt,\\
d\eta^2_t&=a(\eta^1_t,\eta^2_t)dt+\sum_{k=1}^{m}\sigma_{k}(\eta^1_t,\eta^2_t)dw_{k}(t).
\end{aligned}
\right.
\end{equation}
Giving the initial condition  $\eta_0^1=t_0$, $\eta_0^2=x_0$, we denote the corresponding distribution of the process $(\eta_t)_{t\geq 0}$ by $\mathds{P}_{t_0,x_0}$.

The theory of additive functionals can be applied to $(\eta_t)_{t\geq0}$ because it is a homogeneous Markov process.


Let $h$ be a non-negative bounded measurable function on $E=[0,\infty)\times\R$. Then (c.f. Example \ref{Example_general})
$$
A_t=\int_0^t h(\eta_s)ds
$$
is a W-functional of the process $(\eta_t)_{t\geq0}$. Its characteristic is equal to
\begin{multline}\label{eq_character}
f_t(t_0,x_0)=\E_{t_0,x_0}\int_0^t  h(\eta_s)ds= \int_0^t ds\int_{\R}G(t_0,x_0,t_0+s,y)h(t_0+s,y)dy=\
\\
\int_{t_0}^{t_0+t} ds\int_{\R}G(t_0,x_0,s,y)h(s,y)dy,
\end{multline}
where $G(s,x,t,y)$, $0\leq s\leq t, \ x\in\R,y\in \R,$ is the transition probability density of the process $(\eta^2_t)_{t\geq0}$.

Let a measure $\nu$ on $[0,\infty)\times\R$ be such that $\int_{t_0}^{t_0+t} \int_{\R}G(t_0,x_0,s,y)\nu(ds,dy)<\infty$ for all $t\geq 0$, $t_0\geq0$, $x_0\in\R$. If there exists a sequence of non-negative bounded continuous functions  $\{h_n:\ n\geq 1\}$ such that for each $T>0$,
\begin{multline*}
\lim_{n\to\infty}\sup_{t\in[0,T]}\sup_{t_0\in[0,\infty), x_0\in\R}\left|\int_{t_0}^{t_0+t} ds\int_{\R}G(t_0,x_0,s,y)h_n(s,y)dy-\right.\\
\left.\int_{t_0}^{t_0+t} \int_{\R}G(t_0,x_0,s,y)\nu(ds,dy)\right|=0,
\end{multline*}
then by Theorem \ref{Theorem_Convergence_characteristics} there exists a W-functional corresponding to the measure $\nu$ with its
characteristic being equal to $\int_{t_0}^{t_0+t} \int_{\R}G(t_0,x_0,s,y)\nu(ds,dy)$.


\begin{Theorem}[Corollary of Theorem \ref{Theorem_sufficient_condition}]\label{Theorem_sufficient_condition_1}
Let the condition
\begin{equation}\label{Cond_A}
\lim_{t\downarrow 0}\sup_{t_0\in[0,\infty), \ x_0\in\R}f_t(t_0,x_0)=\lim_{t\downarrow 0}\sup_{t_0\in[0,\infty), \ x_0\in\R}\int_{t_0}^{t_0+t} \int_{\R}G(t_0,x_0,s,y)\nu(ds,dy)=0
\end{equation}
hold. Then $f_t(z), z\in[0,\infty)\times\R$, is the characteristic of a W-functional
$A_t^{\nu}$. Moreover,
$$
A_t^{\nu}=\llim_{\varepsilon\downarrow 0}\int_0^t \frac{f_\varepsilon(\eta_u)}{\varepsilon}du,
$$
and the sequence of characteristics of integral functionals $\int_0^t \frac{f_\varepsilon(\eta_u)}{\varepsilon}du$ converges to $f_t(z)$ in sense of the relation \eqref{eq_uni_conv_char}.
\end{Theorem}

Let $\eta_t(t_0,x_0)=(\eta_t^1(t_0,x_0),\eta_t^2(t_0,x_0))$ be a solution of equation (\ref{eq_eta}) starting from the point $(t_0,x_0)$ and defined on a probability space $(\Omega, \mathcal{F}, \mathcal{F}_t, \P)$ . Let  $\P_{t_0,x_0}$ be the distribution of the process $(\eta_t(t_0,x_0))_{t\geq0}$, where $t_0\geq0$, $x_0\in\R$. In Dynkin's notation (see \cite{Dynkin63})
$(\eta_t(t_0,x_0))_{t\geq0},\mathcal{F}_t, \P)$, $t_0\geq0, x_0\in\R$, is called a Markov family of random functions.

Let a measure $\nu$ satisfy the condition of Theorem \ref{Theorem_sufficient_condition_1}. Then there exists a W-functional $A_t^\nu$  of the process $(\eta_t)_{t\geq0}$. According to the definition of  W-functionals, the functional  is measurable w.r.t. $\sigma$-algebra generated by the process $(\eta_t)_{t\geq0}$. Since the process  $(\eta_t)_{t\geq0}$ is continuous and has the infinite life-time, we can consider $A_t^{\nu}=A_t^{\nu}(\cdot)$ as a measurable function on $[0,\infty)\times C([0,\infty),\R)$ that depends only on the behavior of the process on $[0,t]$.
The composition $A_t^\nu(\eta_{\cdot}(t_0,x_0))$, $t\geq 0$, is called a W-functional of $(\eta_t(t_0,x_0))_{t\geq 0}$ corresponding to the measure $\nu$. The function $A_t^\nu(\eta_{\cdot}(t_0,x_0))$ is defined for all $t_0\geq0, x_0\in\R$.

If $t_0=0$, $x_0=x$, the process $\eta_t^2(0,x)=\eta_t^2(x)$ is a solution of equation (\ref{eq_main}) starting from $x$ and therefore $\eta_t^2(x)=\varphi_t(x)$. Then $\eta_t(0,x)=(t,\varphi_t(x))$. Since the first coordinate $\eta_t^1(t_0,x_0)=t_0+t$ is non-random, we denote $A_t^\nu(\eta_{\cdot}(0,x))$ as $A_t^\nu(\varphi_{\cdot}(x))$.

Let us show that  the condition \eqref{Cond_A} can be replaced by a more convenient condition. If $a$ and $\sigma$ are bounded and measurable, and $\sigma$ satisfies condition (C2), then the transition probability density of the process $(\eta_2(t))_{t\geq0}$ satisfies the Gaussian estimates
(see \cite{Aronson67}):
\begin{equation}\label{eq_gaussian_estimates}
\frac{C_1}{t^{d/2}}\exp\left\{-c_1\frac{|y-x|^2}{t-s}\right\}\leq G(s,x,t,y)\leq \frac{C_2}{t^{d/2}}\exp\left\{-c_2\frac{|y-x|^2}{t-s}\right\}
\end{equation}
valid in every domain of the form $0\leq s<t\leq T, x\in\R, y\in\R,$ where $T>0$. Constants $C_1, c_1, C_2, c_2$ are positive and depend only on $d, T,$  $\|a\|_{T,\infty}$,  $\|\sigma_k\|_{T,\infty}, 1\leq k\leq m$, and ellipticity constant $B$, where
$\|a\|_{T,\infty}=\sup_{t\in[0,T]}\sup_{x\in\R}\|a(t,x)\|$.

Denote by $p_0(s,x,t,y)$ the transition probability density of a Wiener process:
\begin{equation}\label{eq_Wiener_density}p_0(s,x,t,y)=\frac{1}{(2\pi(t-s))^{d/2}}\exp\left\{-\frac{|y-x|^2}{2(t-s)}\right\}.
\end{equation}
By analogy with the Kato class (c.f. \cite{Kuwae+07}), we introduce the following definition.
\begin{Definition}
\label{def_Kato_type}
A measure $\nu$  on $[0,\infty)\times\R$ is a measure of  the class $\mathcal{K}$ if
\begin{equation} \label{Cond_A_prime}
\lim_{t\downarrow 0}\sup_{t_0\in[0,\infty), \ x_0\in\R}\int_{t_0}^{t_0+t}\int_{\R}p_0(t_0,x_0,s,y)\nu(ds,dy)=0.
\end{equation}
\end{Definition}
Taking into account (\ref{eq_gaussian_estimates}) it is easy to see that $\nu$ satisfies the condition (\ref{Cond_A}) if and only if it is of the class  $\mathcal{K}$.
\begin{Definition}A signed measure $\nu$ is of the class  $\mathcal{K}$ if the measure $|\nu|$ is of the class  $\mathcal{K}$, where $|\nu|=\nu^++\nu^-$ is the variation of $\nu$; $\nu^+,\nu^-$ are the measures from the Hahn-Jordan decomposition $\nu=\nu^+-\nu^-$.
\end{Definition}
Let $\nu=\nu^+-\nu^-$ be a signed measure belonging to the class $\mathcal{K}$. Then by Theorem \ref{Theorem_sufficient_condition} there exist W-functionals $A_t^{\nu^{\pm}}$. Denote $A_t^{\nu}=A_t^{\nu^+}-A_t^{\nu^-}$.
\begin{Remark}\label{Remark_Hahn_decomp}
Suppose that the signed measure $\nu$ can be represented in the form $\nu=\widetilde{\nu}^+-\widetilde{\nu}^-$, where $\widetilde{\nu}^+$, $\widetilde{\nu}^-$ are of the class $\mathcal{K}$ but  are not necessarily orthogonal. Then one can see that $A_t^{\nu^+}-A_t^{\nu^-}=A_t^{\widetilde\nu^+}-A_t^{\widetilde\nu^-}$.
\end{Remark}

In what follows we will often deal with measures which have densities with respect to the Lebesgue measure on $[0,\infty)\times\R$.
\begin{Definition} A measurable function $h$ on $[0,\infty)\times\R$ is called a function of  the class $\mathcal{K}$ if the signed measure $\nu(ds,dy)=h(s,y)dsdy$ is of the class $\mathcal{K}$.
\end{Definition}
\begin{Remark}\label{Remark_homogeneous_coef}
Let $\nu(ds,dx)=\mu(dx)ds$, where $\mu$ is a measure on $\R$. Then the relation (\ref{Cond_A_prime}) transforms into the following one
\begin{equation}\label{eq_Kato_simple}
\lim_{t\downarrow 0}\sup_{x_0\in\R}\int_{0}^{t}ds\int_{\R}p_0(0,x_0,s,y)\mu(dy)=0.
\end{equation}
It was shown (e.g., Theorem 2.1 in \cite{Chen2002}) that $\mu$ satisfies the condition \eqref{eq_Kato_simple} if and only if
\begin{eqnarray}\label{eq_Kato_class_1}
\sup_{x\in\mathds{R}}\int_{|x-y|\leq1}\mu(dy)<\infty,& \mbox{when} \ d=1;\\
\label{eq_Kato_class_2}\lim_{\varepsilon\downarrow0}\sup_{x\in\mathds{R}^2}\int_{|x-y|\leq\varepsilon}\ln\frac{1}{|x-y|}\mu(dy)=0,& \mbox{when} \ d=2;\\
\lim_{\varepsilon\downarrow0}\sup_{x\in\mathds{R}^d}\int_{|x-y|\leq\varepsilon}|x-y|^{2-d}\mu(dy)=0,& \mbox{when} \ d\geq3.\label{eq_Kato_class_3}
\end{eqnarray}
Consider now a measure $\nu$ of the form $\nu(ds,dx)=\mu(s,dx)ds$. Similarly to \eqref{eq_Kato_class_1}-\eqref{eq_Kato_class_3} one can obtain that if for each $T>0$,  $\mu$ satisfies the condition
\begin{eqnarray}
\sup_{t\in[0,\infty)}\sup_{x\in\mathds{R}}\int_{|x-y|\leq1}\mu(t,dy)<\infty,& \mbox{when} \ d=1,\label{eq_Kato_class_1_prime}\\
\lim_{\varepsilon\downarrow0}\sup_{t\in[0,\infty)}\sup_{x\in\mathds{R}^2}\int_{|x-y|\leq\varepsilon}\ln\frac{1}{|x-y|}\mu(t,dy)=0,& \mbox{when} \ d=2,\label{eq_Kato_class_2_prime}\\
\lim_{\varepsilon\downarrow0}\sup_{t\in[0,\infty)}\sup_{x\in\mathds{R}^d}\int_{|x-y|\leq\varepsilon}|x-y|^{2-d}\mu(t,dy)=0,& \mbox{when} \ d\geq3,\label{eq_Kato_class_3_prime}
\end{eqnarray}
then the measure $\nu$ is of the class $\mathcal{K}$.
\end{Remark}
\begin{Remark}\label{Remark_measure_finite}
Let the measure $\nu(ds,dx)=\mu(s,dx)ds$ satisfy one of the conditions \eqref{eq_Kato_class_1_prime}-\eqref{eq_Kato_class_3_prime}. Then it can be verified (c.f. \cite{Dynkin63}, Lemma 8.3) that for each $T>0$, $r>0$, there exists $K=K(r,T)>0$ such that for all $x\in\R$, $t\in[0,T]$,
$$
\mu(t,B(x,r))<K,
$$
where $B(x,r)$ is the ball with center at $x$ and radius $r$.
\end{Remark}

In the sequel we  use the following modification of Khas'minskii's lemma (see \cite{Khasminskii59} or \cite{Sznitman98}, Ch.1, Lemma 2.1).
 \begin{Lemma} \label{Lemma_exp_moment} Let $A_t$  be a W-functional with the characteristic $f_t$ satisfying the condition (\ref{Cond_A}). Then for all $p>0$, $t\geq0$, there exists a constant $C>0$ depending on $p,t$, and the rate of convergence in (\ref{Cond_A}) such that
 $$
 \sup_{t_0\in[0,\infty), x_0\in\R}\E_{t_0,x_0}\exp\{pA_t\}<C.
 $$
 \end{Lemma}

\begin{Example}\label{Example_continuous}
Let $\nu(dt,dx)=h(t,x)dtdx$, where $h$ is a non-negative bounded measurable function. Then the measure $\nu$ is of the class $\mathcal{K}$.  The functional
$$
A_t:=\int_0^t h(\eta_s)ds
$$
is a W-functional of the process $(\eta_t)_{t\geq0}$ with characteristic defined by \eqref{eq_character}, and
$$
A_t(\varphi(x))=\int_0^t h(s,\varphi_s(x))ds.
$$
\end{Example}
\begin{Example} {\it Local time.}\label{Example_loc_time}
Let $d=1$.  It is well known that for each $x\in\mathds{R}$, $y\in\mathds{R}$ there exists a local time of the process $(\varphi_t(x))_{t\geq0}$ at the point $y$, which is defined by the formula
$$
L_t^{y}(\varphi(x))=\llim_{\varepsilon\downarrow 0}\frac{1}{2\varepsilon}\int_0^t \mathds{1}_{[y-\varepsilon,y+\varepsilon]}(\varphi_s(x))ds.
$$

It can be checked that $L_t^{y}(\varphi(x))$ is a W-functional of $(\varphi_t(x))_{t\geq0}$ corresponding to the measure $\nu(ds,dx)=ds\delta_y(dx)$, where $\delta_y$ is the delta measure at the point $y$.
Indeed,
for fixed $y\in\mathds{R}$ and each $\varepsilon>0$, put
$$
h^{\varepsilon,y}(t,x)=h^{\varepsilon,y}(x)=\frac{1}{2\varepsilon}\mathds{1}_{[y-\varepsilon,y+\varepsilon]}(x), t\geq 0, x\in \mathds{R},
$$
and $\nu^{\varepsilon,y}(dt,dx)=h^{\varepsilon,y}(t,x)dtdx$. The function $h^{\varepsilon,y}$ is bounded and measurable. Then (see Example \ref{Example_continuous}) there exists a W-functional of the process $(\eta_t)_{t\geq0}$ corresponding to the measure $\nu^{\varepsilon,y}$. This functional  is   defined by the formula
$$
A_t^{\varepsilon,y}:=A_t^{\nu^{\varepsilon,y}}=\int_{0}^{t}h^{\varepsilon,y}(\eta_s)ds=\frac{1}{2\varepsilon}\int_{0}^{t} \mathds{1}_{[y-\varepsilon,y+\varepsilon]}(\eta_s^2)ds
$$
and its characteristic is
equal to
$$
f_t^{\varepsilon,y}(t_0,x_0)=\E_{t_0,x_0}A_t^{{\varepsilon,y}}(\eta)=\int_{t_0}^{t_0+t}ds\int_{\R}
G(t_0,x_0,s,v)h^{\varepsilon,y}(s,v)dv.
$$
One can see that $f_t^{\varepsilon,y}(t_0,x_0)$ tends to $$
f_t^y(t_0,x_0)=\int_{t_0}^{t_0+t}G(t_0,x_0,s,y)ds=\int_{t_0}^{t_0+t}ds\int_{\R}G(t_0,x_0,s,v)\delta_y(dv)
$$
as $\varepsilon \to 0$ uniformly in $t\in[0,T], t_0\in [0,\infty], x_0\in\mathds{R}$. Then by Theorem \ref{Theorem_Convergence_characteristics}  there exists a functional
$$
A_t^{y}=\llim_{\varepsilon\downarrow 0}A_t^{\varepsilon,y}=\llim_{\varepsilon\downarrow 0}\int_{0}^{t}h^{\varepsilon,y}(\eta_s)ds.
$$
In particular,
$$
A_t^{y}(\varphi_t(x))=L_t^{y}(\varphi(x)).
$$

Note that if $d\geq 2$, the measure $\delta_y$ is not of the class $\mathcal{K}$. This agrees with the well-known fact that the local time for a multidimensional Wiener process does not exist.
%

\end{Example}
The following  lemma deals with the convergence of W-functionals of, generally speaking, different random functions.
\begin{Lemma}\label{Lemma_converg_functionals}
Let $\{(\xi_{n,t})_{t\geq0}: n\geq 0\}$ be a sequence of homogeneous Markov random functions defined on a common probability space  $(\Omega,\mathcal F, P)$ with the common phase space $(E,\mathcal{B})$, where $E$ is a metric space, $\mathcal{B}$ is the Borel $\sigma$-algebra. For $n\geq 0$, let $A_{n,t}=A_{n,t}(\xi_n)$ be a W-functional of the random function $(\xi_{n,t})_{t\geq0}$ with the characteristic $f_{n,t}(z)$.

Assume that
\begin{enumerate}[1)]
\item for each $t\geq 0$, $f_{0,t}(z)$ is continuous in $z\in E$;
\item for each $t\geq 0$, $\xi_{n,t}\to\xi_{0,t}$, $n\to\infty$, {in probability}   $P$;
\item for all $n\geq 0$,
$$
\lim_{\delta\downarrow 0}\|f_{n,\delta}\|_E=0,
$$
where
$$
\|f_{n,\delta}\|_E=\sup_{z\in E}|f_{n,\delta}(z)|;
$$
\item for each $t>0$, $\|f_{n,t}-f_{0,t}\|_{E}\to 0$, $n\to\infty$.
\end{enumerate}
Then for each $T>0$,
$$
\sup_{t\in[0,T]}|A_{n,t}(\xi_n)-A_{0,t}(\xi_0)|\to 0, \ n\to\infty, \ \mbox{in probability} \  P.
$$
\end{Lemma}
\begin{proof}
Note that $A_{n,t}^\delta:=\frac1\delta \int_0^tf_{n,\delta}(\xi_{n,s})ds$ is a W-functional of the process $(\xi_{n,t})_{t\geq0}$. Denote its characteristic by $f_{n,t}^\delta$. Then by \cite{Dynkin63}, Lemma 6.5, for all $t\geq 0, z\in E$,
$$
\E_z\left(A_{n,t}-\frac1\delta \int_0^tf_{n,\delta}(\xi_{n,s})ds\right)^2\leq 2\left(f_{n,t}(z)+f_{n,t}^\delta(z)\right)\sup_{0\leq u\leq t}\|f_{n,u}-f_{n,u}^\delta\|_{E}.
$$
Similarly to the proof of \cite{Dynkin63}, Theorem 6.6, we get
\begin{multline*}
|f_{n,t}^\delta(z)-f_{n,t}(z)|\leq \\
\frac 1\delta \int_{t}^{t+\delta}|f_{n,u}(z)-f_{n,t}(z)|du
+\frac 1\delta\int_0^\delta f_{n,u}(z)du
\leq\\
\frac 1\delta\int_t^{t+\delta}\|f_{n,u-t}\|_Edu+\frac 1\delta\int_0^\delta\|f_{n,u}\|_Edu\leq 2\|f_{n,\delta}\|_{E}.
\end{multline*}
So for all $t\geq 0$,
\begin{equation}\label{eq_alfa}
\sup_{0\leq u\leq t}\|f_{n,u}^\delta-f_{n,u}\|_{E}\leq 2\|f_{n,\delta}\|_{E}.
\end{equation}
Using the calculations of the proof of \cite{Dynkin63}, Theorem 6.6, once more we obtain
\begin{multline}\label{eq_beta}
f_{n,t}(z)+f_{n,t}^\delta(z)=f_{n,t}(z)+\frac 1\delta\int_t^{t+\delta}f_{n,u}(z)du-\frac1\delta\int_0^\delta f_{n,u}(z)du\leq\\ \|f_{n,t}\|_E+\|f_{n,t+\delta}\|_E\leq 2\|f_{n,t+\delta}\|_E.
\end{multline}
The inequalities \eqref{eq_alfa} and \eqref{eq_beta} give us the relation
\begin{equation}\label{abcd}
\E_z\left(A_{n,t}(\xi_n)-\frac1\delta \int_0^tf_{n,\delta}(\xi_{n,s})ds\right)^2\leq 8\|f_{n,\delta}\|_{E}\|f_{n,t+\delta}\|_E.
\end{equation}

Further, we have
\begin{multline}\label{4abcd}
\E_z(A_{n,t}(\xi_n)-A_{0,t}(\xi_0))^2\leq
4\left[\E_z\left(A_{n,t}(\xi_n)-\frac1\delta \int_0^tf_{n,\delta}(\xi_{n,s})ds\right)^2\right.+\\
\E_z\left(\frac1\delta \int_0^t f_{n,\delta}(\xi_{n,s})ds-\frac 1\delta\int_0^tf_{0,\delta}(\xi_{n,s})ds\right)^2+\\
\E_z\left(\frac 1\delta\int_0^t f_{0,\delta}(\xi_{n,s})ds-\frac1\delta \int_0^t f_{0,\delta}(\xi_{0,s})ds\right)^2+\\
\left.\E_z\left(\frac1\delta \int_0^t f_{0,\delta}(\xi_{0,s})ds-A_{0,t}(\xi_0)\right)^2\right]=
4[I+II+III+IV].
\end{multline}

For any $\varepsilon>0$, by assumption 3) we can choose $\delta>0$ such that  $\|f_{0,\delta}\|_E< \varepsilon$.
According to 4) there exists $n_0>0$ such that for all $n>n_0$,
$$
\|f_{n,\delta}-f_{0,\delta}\|_{E}<\varepsilon.
$$
Then for all $n>n_0$,
$$
\|f_{n,\delta}\|_{E}\leq\|f_{n,\delta}-f_{0,\delta}\|_{E}+\|f_{0,\delta}\|_{E}<2\varepsilon.
$$
Notice that for each $n\geq0$, $k\geq 1$,
$\|f_{n,k\delta}\|_{E}\leq k\|f_{n,\delta}\|_{E}$. This implies that for any $t\geq 0$,
$M_t:=\sup_{n\geq0}\|f_{n,t}\|_{E}<\infty$.
Taking into account (\ref{abcd}), we obtain that for all $n>n_0$,
$$
I\leq 16M_{t+\delta}\varepsilon,
$$ and the same estimate holds for $IV$.

By the H\"older inequality,
$$
II\leq\frac{t}{\delta^2}\E_z\int_0^t\left(f_{n,\delta}(\xi_{n,s})-f_{0,\delta}(\xi_{n,s})\right)^2ds\leq \frac{t^2}{\delta^2}\E_z\sup_{z\in E} \left(f_{n,\delta}(z)-f_{0,\delta}(z)\right)^2
$$
The assumptions  4) yields the estimate
$
II\leq \varepsilon
$
valid for all $n\geq n_1=n_1(\varepsilon,\delta)$.

Similarly,
$$
III\leq\frac{t}{\delta^2}\E_z\int_0^t\left(f_{0,\delta}(\xi_{n,s})-f_{0,\delta}(\xi_{0,s})\right)^2ds.
$$
The continuity of the function $f_{0,t}(\cdot)$ and assumption 2) provide the convergence
$f_{0,\delta}(\xi_{n,s})$ to $f_{0,\delta}(\xi_{0,s})$ as $n$ tends to $\infty$ in probability. This convergence together with 3) allow us to use the  dominated convergence theorem and prove that $III\to 0$ as $n\to\infty$.  Then the right-hand side of \eqref{4abcd} tends to $0$ as $n$ tends to $\infty$. The uniform convergence follows from  Proposition \ref{Proposition_uniform_convergence}. This completes the proof.
\end{proof}

\section{The main result}\label{Section_main_result}
The main result on differentiability with respect to the initial data of a flow generated by equation (\ref{eq_main})  is given in the following theorem.
\begin{Theorem}\label{Theorem_main}
Let measurable bounded function $a=(a^1,\dots,a^d):[0,\infty)\times\R\to\R$ be  such that for each $t\geq0$ and all $1\leq i \leq d,$ $a^i(t,\cdot)$ is a function of bounded variation on $\R$, i.e., for each $1\leq j\leq d$, the generalized derivative $\mu^{ij}(t,dy)=\frac{\partial a^i}{\partial y_j}(t,dy)$ is a signed measure on $\R$ . Assume that the signed measures $\nu^{ij}(dt,dy):=\mu^{ij}(t,dy)dt, \ 1\leq i,j\leq d,$ are of  the class $\mathcal{K}$. Let $\sigma:[0,\infty)\times\R\to\R\times\mathds{R}^m$ be a bounded continuous function satisfying (C1), (C2), and the following conditions
\begin{enumerate}[(C3)]
\item {\it H\"older continuity}: For each $T>0$, there exist $L>0$, $0<\alpha\leq 1$ such that for all  $t_1,t_2\in[0,T]$, $x_1,x_2\in\R$, $1\leq i\leq d$, $1\leq k\leq m$,
$$
|\sigma_k^{i}(t_1,x_1)-\sigma_k^{i}(t_2,x_2)|\leq L\left(|t_1-t_2|^{\alpha/2}+|x_1-x_2|^\alpha\right).
$$
\end{enumerate}
\begin{enumerate}[(C4)] \label{cond_C4}
\item There exists $\r>0$ such that  for all $1\leq k\leq m$, $1\leq i,j\leq d$, the function $\left|\frac{\partial \sigma_k^i}{\partial y_j}(s,y)\right|^{2+\r}$ belongs to the class $\mathcal{K}$.
\end{enumerate}

Then there exists the derivative $Y_t(x)=\nabla\varphi_t(x)$ in $L_p$-sense: for all $p>0$,  $x\in\R$, $v\in\R$, $t\geq0$,
\begin{equation}\label{eq_derivative_main2}
\mathds{E}\left|\frac{\varphi_t(x+\varepsilon v)-\varphi_t(x)}{\varepsilon}-Y_t(x)v\right|^p\to 0, \ \varepsilon\to 0.
\end{equation}

The derivative is a unique solution of the integral equation
\begin{equation}\label{eq_derivative_main}
Y_t(x)=E+\int_0^t dA_s^{\nu}(\varphi(x))Y_s(x)+\sum_{k=1}^{m}\int_0^t\nabla\sigma_{k}(s,\varphi_s(x))Y_s(x)dw_{k}(s),
\end{equation}
where $E$ is the $d\times d$-identity matrix,  $\nabla\sigma_{k}(s,y)=\left(\frac{\partial\sigma_{k}^{i}}{\partial y_j}(s,y)\right)_{1\leq i,j\leq d}$; the first integral in the right-hand side of (\ref{eq_derivative_main}) is the Lebesgue-Stieltjes  integral with respect to the continuous function of bounded variation $t\rightarrow A_t^{\nu}(\varphi(x))$.

Moreover,
\begin{equation}\label{eq_Sobolev_derivative}
P\left\{\forall t\geq0: \varphi_t(\cdot)\in W_{p,loc}^1(\R,\R), \nabla\varphi_t(x)=Y_t(x) \ \mbox{for} \ \lambda\mbox{-a.a.} \ x\right\}=1,
\end{equation}
where $\lambda$ is the Lebesgue measure on $\R$.
\end{Theorem}
\begin{Remark} The W-functional $A_t^\nu=\left(A_t^{\nu^{ij}}\right)_{1\leq i,j\leq d}$ is well defined because the signed measure $\nu$ is of the class $\mathcal{K}$.
\end{Remark}
\begin{Remark} Recall that for all $1\leq i,j\leq d$, the mappings $A_t^{\nu^{ij,\pm}}$, which we will denote by  $A_t^{ij,\pm}$, are continuous and monotonous in $t$. So for each $T>0$, the function $t\to A_t^{ij}$ is a continuous function of bounded variation on $[0,T]$ almost surely.
\end{Remark}

\section{The proof of Theorem \ref{Theorem_main}}\label{Proof of Theorem_main}

The existence and uniqueness of solution for equation (\ref{eq_derivative_main}) follows from \cite{Protter04}, Ch. V, Theorem 7. Indeed, condition (C4) provides that for all $1\leq k\leq m$, $\int_0^t|\nabla\sigma_k(s,\varphi_s(x))|^2ds<\infty$ a.s. and consequently
$$
\int_0^t dA_s^{\nu}(\varphi(x))+\sum_{k=1}^{m}\int_0^t\nabla\sigma_{k}(s,\varphi_s(x))dw_{k}(s), t\geq 0,
$$
is a semimartingale.

It is well known that the statement of the theorem is true in the case of smooth coefficients, and the derivative satisfies equation \eqref{eq_deriv_new}. To prove the theorem in general case we approximate the initial equation by equations with smooth coefficients.

The proof is divided into two steps.
 \subsection {}\label{subsction_compact_case}  In the first step, we assume that there exists $R>0$ such that for all $t\geq0$,  $x\in\R$,   $|x|\geq R$, $a(t,x)=0$,  $\sigma(t,x)=\widetilde{\sigma}=const$, $\widetilde{\sigma}\widetilde{\sigma}^\ast>0$.

For $n\geq1,$ let $\omega_n\in C_0^{\infty}(\R)$ be a non-negative function such that
$\int_{\mathds{R}^d}\omega_n(z)dz=1$, and $\omega_n(x)=0, \ |x|\geq 1/n$.
For all $t\geq 0$, $x\in\R$, $n\geq 1$, and $1\leq k\leq m$, put
\begin{eqnarray}\label{eq_a_n}
& a_n(t,x)=(\omega_n\ast a)(t,x)=\int_{\mathds{R}^d} \omega_n(x-y)a(t,y)dy, \\
& \sigma_{n,k}(t,x)=(\omega_n\ast \sigma_k)(t,x)=\int_{\mathds{R}^d} \omega_n(x-y)\sigma_{k}(t,y)dy.
\end{eqnarray}
Note that for each $T>0$,
\begin{eqnarray}\label{eq_ineq_for_norms}
\sup_{n\geq1}\|a_n\|_{T,\infty}\leq \|a\|_{T,\infty},\\
\sup_{n\geq1}\|\sigma_{n,k}\|_{T,\infty}\leq \|\sigma_k\|_{T,\infty}, \ 1\leq k\leq m,
\end{eqnarray}
where
$$
\|a\|_{T,\infty}=\sup_{t\in[0,T]}\sup_{x\in\R}|a(t,x)|.
$$
Besides, for all $n\geq1$, $\sigma_n$ satisfies (C2), and the ellipticity constant can be chosen uniformly in $n$.
\begin{Remark}\label{Remark_uniform_Gaussian_estimates}
For all $n\geq1$ the transition probability density of the process $(\varphi_{n,t}(x))_{t\geq0}$ satisfies the inequality \eqref{eq_gaussian_estimates}. It follows from \eqref{eq_ineq_for_norms} and (C2), which holds uniformly in $n$, that the constants in \eqref{eq_gaussian_estimates} can be chosen uniformly in $n\geq 1$.
\end{Remark}
For each $T>0$, we have $a_n\to a$,  $n\to\infty$, in
$L_{1}([0,T]\times\R).$ Passing to
subsequences we may assume without loss of generality that
$a_n(t,x)\to a(t,x)$, $n\to\infty,$ for almost all $t\geq0$ and almost all $x$ w.r.t. the
Lebesgue measure.
Then for all $n\geq 1, \ t\geq0$, $x\in\R$ such that $|x|\geq R+1,$
\begin{equation*}
a_n(t,x)=0, \ \sigma_{n}(t,x)=\widetilde{\sigma}.
\end{equation*}
Without loss of generality we can  suppose that  this is true for all $x$ such that $|x|>R$.
Moreover, from (C3) we can conclude that for each $T>0$, $\sigma_{n}\to \sigma$, $n\to\infty$, uniformly in $(t,x)\in[0,T]\times\R$.

Consider the SDE
\begin{equation}\label{eq_main_n} \left\{
\begin{aligned}
d\vp_{n,t}(x)&=a_n(t,\vp_{n,t}(x))dt+\sum_{k=1}^{m}\sigma_{n,k}(t,\vp_{n,t}(x))dw_{k}(t),\\
\vp_{n,0}(x)&=x, \ x\in\mathds{R}^d.
\end{aligned}\right.
\end{equation}
For each $n\geq 1$ there exists a unique strong solution of equation (\ref{eq_main_n}).
\begin{Lemma}\label{Lemma_Converg_Solutions} {\it For each $p\geq1$,
\begin{enumerate}[1)]
\item for all $t\geq0$ and any compact set $U\in\mathds{R}^d$,
$$
\sup_{x\in U, \ n\geq
1}\left(\mathds{E}(|\vp_{n,t}(x)|^p+|\vp_t(x)|^p)\right)<\infty;
$$
\item for all   $x\in\mathds{R}^d, \ T\geq0,$
$$
\mathds{E}\left(\sup_{0\leq t\leq
T}|\vp_{n,t}(x)-\vp_t(x)|^p\right)\to 0 \ \mbox{as} \ n\to \infty.
$$
\end{enumerate} }
\end{Lemma}
\begin{proof}
The first statements follows from the uniform boundedness of the coefficients, the second one is a consequence of
\cite{Luo11}, Theorem 3.4.
\end{proof}

For $n\geq 1$, put $\nabla a
_n=\left(\frac{\partial a_n^i}{\partial x_j}\right)_{1\leq i,j \leq d}$, $\nabla \sigma_{n,k}=\left(\frac{\partial \sigma_{n,k}^i}{\partial x_j}\right)_{1\leq i,j \leq d}$. Denote by $Y_{n,t}(x)$ the
matrix of derivatives of $\varphi_{n,t}(x)$ in $x$, i.e.,
$Y_{n,t}^{ij}(x)=\frac{\partial\varphi_{n,t}^i(x)}{\partial x_j}$, $1\leq i,j\leq d$.
Then $Y_{n,t}(x)$ satisfies the equation
\begin{equation}\label{eq_derivative}
Y_{n,t}(x)=E+\int_0^t\nabla a_n(s,\vp_{n,s}(x))Y_{n,s}(x)ds+\sum_{k=1}^{m}\int_0^t\nabla \sigma_{n,k}(s,\vp_{n,s}(x))Y_{n,s}(x)dw_{k}(s),
\end{equation}
where $E$ is the $d$-dimensional identity matrix.

By the properties of convolution of a
generalized function (see \cite{Vladimirov67}, Ch. 2, \S7),
\begin{equation}\label{eq_convolution}
\nabla a_n=\nabla a\ast \omega_n=a\ast \nabla \omega_n, \ n\geq 1.
\end{equation}
Note that for all $n\geq 1$, $1\leq i,j\leq d$, $\nabla a_n^{ij}$ is a bounded measurable function on $[0,\infty)\times\R$.
Then (see Example 1) there exists a continuous homogeneous additive functional
$$
A_{n,t}^{ij}(\varphi_n(x))=\int_0^t\frac{\partial a_n^i}{\partial y_j}(s,\varphi_{n,s}(x))ds
$$
corresponding to the signed measure $\nabla a_n^{ij}(s,y)dsdy$.

Denote $\mu_n^{ij}(t,y)dy=\frac{\partial a_n^i}{\partial y_j}(t,y)dy$. For each $n\geq1$, $1\leq i,j \leq d,$ put $\mu_n^{ij,\pm}=\mu^{ij,\pm}\ast\omega_n$  (recall that $\mu^{ij}(t,dy)=\frac{\partial a^i}{\partial y_j}(t,dy)$). Then $\mu_n^{ij}=\mu_n^{ij,+}-\mu_n^{ij,-}$. It can be easily seen that the measures $\nu_n^{ij,\pm}(dt,dy)=\mu_n^{ij,\pm}(t,dy)dt$, $n\geq 1$, are of the class $\mathcal{K}$.  By Remark \ref{Remark_uniform_Gaussian_estimates}, for each $x\in\R$ there exist W-functionals $A_{t}^{\nu_n^{ij,\pm}}(\varphi_{n,\cdot}(x))$, which we will denote by $A_{t}^{ij,\pm}(\varphi_{n}(x))$.
Generally speaking, $\mu_n^{ij,\pm}\neq (\mu^{ij}\ast\omega_n)^\pm$ but, by Remark \ref{Remark_Hahn_decomp},
$$
A_{n,t}^{ij}(\varphi_n(x))=A_{t}^{\nu_n^{ij}}(\varphi_{n,\cdot}(x))=A_{t}^{ij,+}(\varphi_{n}(x))-A_{t}^{ij,-}(\varphi_{n}(x)).
$$

Denote $\varphi_{0,t}(x)=\varphi_t(x)$, $Y_{0,t}(x)=Y_t(x)$, $a_0=a$, $\sigma_0=\sigma$, $A_{0,t}=A_t$.
\begin{Lemma}\label{Proposition_moments_A}
For all $t\geq0$, $p>0,$ $1\leq i,j \leq d,$ there exists a constant $C$ such that
\begin{equation}\label{eq_exp_Moment_W-functional}
\sup_{ n\geq0}\sup_{x\in\R}\E\exp\left\{pA_{n,t}^{ij,\pm}(\varphi_n(x))\right\}<C.
\end{equation}
\end{Lemma}
\begin{proof}
The statement of lemma follows from Lemma \ref{Lemma_exp_moment} and Remark \ref{Remark_uniform_Gaussian_estimates}.
\end{proof}

\begin{Lemma}\label{Lemma_Converg_Derivatives}
For all $T\geq0$, $x\in\R$, $p>0$,
$$
\sup_{n\geq 0}\mathds{E}\sup_{0\leq t\leq T}|Y_{n,t}(x)|^p<\infty.
$$
\end{Lemma}
\begin{proof} For all $t>0$, $n\geq 0$, define the variation of $A_{n,\cdot}^{ij}$ on $[0,t]$ by
$$
\Var A_{n,t}^{ij}(\varphi(x)):=A_{n,t}^{ij,+}(\varphi(x))+A_{n,t}^{ij,-}(\varphi(x)),
$$
and denote
$$
\Var A_{n,t}(\varphi(x)):=\Sigma_{1\leq i,j \leq d}\Var A_{n,t}^{ij}(\varphi(x)).
$$
Set
$$
\tau_n^N=\inf\left\{t\geq0:\int_0^t\sum_{k=1}^m|\nabla\sigma_{n,k}(s,\varphi_{n,s}(x))|^2ds+\Var A_{n,s}(\varphi(x))+|Y_{n,s}(x)|^2\geq N\right\}.
$$
For the sake of brevity, denote
$$
h_n(t,C,l)=-2l\Var A_{n,t}(\varphi(x))-C\sum_{k=1}^m\int_0^t|\nabla\sigma_k(s,\varphi_s(x))|^2ds.
$$
By Ito's formula, for all $n\geq0$, $l\in\mathds N$,
\begin{multline}\label{eq_moment_1}
e^{h_n(t\wedge\tau_n^N,C,l)}|Y_{n,t\wedge\tau_n^N}(x)|^{2l}=
|Y_{n,0}(x)|^{2l}-
2l\int_0^{t\wedge\tau_n^N}e^{h_n(s,C,l)}|Y_{n,s}(x)|^{2l}d\Var A_{n,s}(\varphi(x))-\\
C\int_0^{t\wedge\tau_n^N}e^{h_n(s,C,l)}|Y_{n,s}(x)|^{2l}\sum_{k=1}^m|
\nabla\sigma_{n,k}(s,\varphi_s(x))|^2ds+\\
2l\int_0^{t\wedge\tau_n^N}e^{h_n(s,C,l)}|Y_{n,s}(x)|^{2l-2}\sum_{i,j=1}^d Y_{n,s}^{ij}(x)\sum_{r=1}^d dA_{n,s}^{ir}(\varphi(x))Y_{n,s}^{rj}(x)+\\
2l\int_0^{t\wedge\tau_n^N}e^{h_n(s,C,l)}|Y_{n,s}(x)|^{2l-2}\sum_{i,j=1}^dY_{n,s}^{ij}(x)\sum_{k=1}^m
\sum_{r=1}^d\nabla\sigma_{n,k}^{ir}(s,\varphi_s(x))Y_{n,s}^{rj}(x)dw_k(s)+\\
2l\int_0^{t\wedge\tau_n^N}e^{h_n(s,C,l)}|Y_{n,s}(x)|^{2l-4}\Big(\sum_{i,j=1}^d\sum_{v,q=1}^d\big(2(l-1)Y_{n,s}^{ij}(x)Y_{n,s}^{vq}(x)
+|Y_{n,s}(x)|^2\delta_{vi}\delta_{qj}\big)\times\\
\sum_{k=1}^m\sum_{r=1}^d\nabla\sigma_{n,k}^{vr}(s,\varphi_{n,s}(x))Y_{n,s}^{rq}(x)\sum_{e=1}^d\nabla\sigma_{n,k}^{ie}(s,\varphi_{n,s}(x))Y_{n,s}^{ej}(x)\Big)ds.
\end{multline}
Here $|\cdot|$ is the Hilbert-Schmidt norm. Note that  the absolute value of the third integral in the right-hand side of (\ref{eq_moment_1}) is less than or equal to that  of the first one. There exists a constant $\widetilde C=\widetilde C(d)>0$ such that the last integral does not exceed
$$
2l\widetilde C\int_0^{t\wedge\tau_n^N}e^{h_n(s,C,l)}|Y_{n,s}(x)|^{2l}\sum_{k=1}^m|\nabla\sigma_{n,k}(s,\varphi_{n,s}(x))|^2ds.
$$
Then we can choose $C>0$ so large that the absolute value of the last integral is less than or equal to  the second integral. We obtain
\begin{equation}\label{eq_h_and_martingale}
e^{h_n(t\wedge\tau_n^N,C,l)}|Y_{n,t\wedge\tau_n^N}(x)|^{2l}\leq|Y_{n,0}(x)|^{2l}+M(t\wedge\tau_n^N),
\end{equation}
where
\begin{multline*}
M(t\wedge\tau_n^N)=\\
2l\int_0^{t\wedge\tau_n^N}e^{h_n(s,C,l)}|Y_{n,s}(x)|^{2l-2}\sum_{i,j=1}^dY_{n,s}^{ij}(x)\sum_{k=1}^m
\sum_{r=1}^d\nabla\sigma_{n,k}^{ir}(s,\varphi_{n,s}(x))Y_{n,s}^{rj}(x)dw_k(s), \\ t\geq 0,
\end{multline*}
is a square integrable martingale.
Then, for all $t\geq 0$,
$$
\E e^{h_n(t\wedge\tau_n^N,C,l)}|Y_{n,t\wedge\tau_n^N}(x)|^{2l}\leq K,
$$
where $K=|Y_{n,0}(x)|^{2l}=|E|^{2l}=d^l$. Passing to the limit as $N\to\infty$, we get that for all $T>0$ there exists $C=C(l,d)$ such that
\begin{equation}\label{eq_finit_expect}
\sup_{n\geq0}\sup_{t\in[0,T]}\E e^{h_n(t,C,l)}|Y_{n,t}(x)|^{2l}\leq K.
\end{equation}
By \eqref{eq_h_and_martingale}, for all $T>0$,
\begin{multline*}
\E\sup_{t\in[0,T]}e^{2h_n(t,C,l)}|Y_{n,t}(x)|^{4l}\leq2\E\sup_{t\in[0,T]}\left(|Y_{n,0}(x)|^{4l}+M^2(t)\right)\leq\\
K^\prime\left(1+\sum_{k=1}^m\E\int_0^Te^{2h_n(s,C,l)}|Y_{n,s}(x)|^{4l}|\nabla\sigma_{n,k}(s,\varphi_{n,s}(x))|^2ds\right).
\end{multline*}
Making use of H\"older's inequality with $p=1+\frac{\rho}{2}$, we get
\begin{multline}\label{eq_Holder_h}
\E\sup_{t\in[0,T]}e^{2h_n(t,C,l)}|Y_{n,t}(x)|^{4l}\leq
K^{\prime}\left[1+\left(\E\int_0^T\left(e^{2h_n(s,C,l)}|Y_{n,s}(x)|^{4l}\right)^{\frac{2+\rho}{\rho}}ds\right)^{\frac{\rho}{2+\rho}}\right.\times\\
\left.\sum_{k=1}^m\left(\E\int_0^T|\nabla\sigma_{n,k}(s,\varphi_{n,s}(x))|^{2+\rho}ds\right)^{\frac{2+\rho}{2}}\right].
\end{multline}

Since for all $1\leq k\leq m$, $1\leq i,j\leq d$, the function $\left|\frac{\partial \sigma_k^i}{\partial y_j}(s,y)\right|^{2+\r}$ is of  the class $\mathcal{K}$, the functions $\left|\frac{\partial \sigma_{n,k}^i}{\partial y_j}(s,y)\right|^{2+\r}$, $n\geq 1$, are of the class $\mathcal{K}$ too. It follows from Lemma \ref{Proposition_moments_A} that for each $T>0$,
\begin{equation}\label{eq_gnom_2}
\sup_{n\geq 1}\E \exp\left\{\int_0^T|\nabla\sigma_{n,k}(s,\varphi_{n,s}(x))|^{2+\rho}\right\}ds<C(T),
\end{equation}
where $C(T)$ is a constant which depends on $T$.
Consequently,
\begin{equation}\label{eq_333}
\sup_{n\geq0}\E\int_0^T|\nabla\sigma_{n,k}(s,\varphi_{n,s}(x))|^{2+\rho}ds<\infty.
\end{equation}
By (\ref{eq_finit_expect}) we have
\begin{equation}\label{eq_gnom}
\sup_{n\geq0}\E\int_0^T\left(e^{2h_n(s,C,l)}|Y_{n,s}(x)|^{4l}\right)^{\frac{2+\rho}{\rho}}ds<\infty.
\end{equation}
From \eqref{eq_333} and \eqref{eq_gnom} we get
\begin{equation}\label{eq_lll}
\sup_{n\geq0}\E\sup_{t\in[0,T]}e^{2h_n(t,C,l)}|Y_{n,t}(x)|^{4l}<\infty.
\end{equation}
Finally, for any $T>0$, by the H\"older inequality,
\begin{multline*}
\sup_{n\geq0 }\E{\sup_{t\in[0,T]}}|Y_{n,s}(x)|^{2l}=\sup_{n\geq0 }\E{\sup_{t\in[0,T]}}\left[\left(e^{h_n(t,C,l)}|Y_{n,s}(x)|^{2l}\right)e^{-h_n(t,C,l)}\right]\leq\\
\sup_{n\geq0 }\left[\left(\E\sup_{t\in[0,T]}e^{2h_n(t,C,l)}|Y_{n,t}(x)|^{4l}\right)^{1/2}\times\right.\\  \left.
\left(\E \exp\left\{4l\Var A_{n,T}(\varphi_n(x))+2C\sum_{k=1}^m\int_0^T|\nabla\sigma_{n,k}(s,\varphi_{n,s}(x))|^2ds\right\}\right)^{1/2}\right].
\end{multline*}
Now the assertion of the lemma follows from \eqref{eq_333}, (\ref{eq_lll}), and the fact that for each $T>0$, $\sup_{n\geq0}\Var A_{n,T}(\varphi_n(x))<\infty$, which is a consequence of Lemma \ref{Proposition_moments_A}.
\end{proof}
\begin{Lemma}\label{Lemma_converg_W_functionals}
For each $T>0$, $x\in\R$, $1\leq i,j\leq d$,
$$
\sup_{0\leq t\leq T}|A_{n,t}^{ij,\pm}(\varphi_n(x))-A_{t}^{ij,\pm}(\varphi(x))|\to 0, \ n\to\infty, \ \mbox{in probability} \ \P.
$$
\end{Lemma}
\begin{proof}
To prove the lemma we make use of Lemma \ref{Lemma_converg_functionals} in which we put $\xi_{n,t}=\eta_{n,t}$, $A_{n,t}=A_{n,t}(\eta_{n})$, $\xi_{0,t}=\eta_t$, and $A_{0,t}=A_{t}(\eta)$, $n\geq1$, $t\geq0$. Here $(\eta_{n,t})_{t\geq 0}$ is a solution to the system of the form (\ref{eq_eta}) with coefficients $a_n, \sigma_{n,k}$. Then
$$
f_{0,t}(t_0,x_0)=\int_{t_0}^{t+t_0} ds\int_{\R}G(t_0,x,s,y)\mu(dy),
$$
where $G(s,x,t,y)$, $0\leq s\leq t$, $x,y\in\R$, is the transition probability density of the process $(\eta_t^2)_{t\geq0}$. For each $T>0$, the function $G(s,x,t,y)$ is continuous on $0\leq s<t\leq T$, $x,y\in \R$ (see \cite{Portenko90}, Ch.2, \S 2). Taking into account the inequality (\ref{eq_gaussian_estimates}), which holds locally uniformly in $x$, we  obtain assertion 1) of the Lemma \ref{Lemma_converg_functionals} from the dominated convergence theorem. Assertion 2) is a consequence of Lemma \ref{Lemma_Converg_Solutions}. Assertion 3) is obvious. Assertion 4) follows from Lemma \ref{Lemma_Converg_Character}, which is proved in Section \ref{Section_Conv_Densities}.
\end{proof}
\begin{Lemma}\label{Lemma_Converg_Derivatives_2}
For all $T\geq0, \ x\in\R,$
$$
\sup_{0\leq t\leq T}|Y_{n,t}(x)-Y_t(x)|\to 0, \ n\to\infty, \ \mbox{in probability} \ \P.
$$
\end{Lemma}
To prove the lemma we need three auxiliary propositions.
The first one is a variant of the Gronwall inequality
and can be obtained by a standard argument.
\begin{Proposition}\label{Prop_Gronwall_Lemma}
Let  $x(t)$ be a continuous function on
$[0,+\infty)$, $C(t)$ be a non-negative continuous function
on $[0,+\infty)$, $K(t)$ be a non-negative, non-decreasing function, and  $K(0)=0$. If for
all $0\leq t\leq T$,
$$
x(t)\leq C(t)+\left|\int_0^t x(s)dK(s)\right|,
$$
then
$$
x(T)\leq \left(\sup_{0\leq t\leq T}C(t)\right)\exp\{K(T)\}.
$$
\end{Proposition}

The following simple proposition is technical.
\begin{Proposition}\label{Proposition_Monot_convergence} Let
$\{h_n : \ n\geq 1\}$ be a sequence of continuous monotonic
functions on $[0,T]$, and $f\in C([0,T]).$ Suppose that  $t\in[0,T],$
$h_n(t)\to h_0(t),$ as $n\to\infty,$ $t\in[0,T]$.
 Then
$$
\sup_{t\in[0,T]}\left|\int_0^t f(s)dh_n(s)-\int_0^t
f(s)dh_0(s)\right|\to 0, \ n\to\infty.
$$
\end{Proposition}
\begin{Proposition}\label{Proposition_Kulik}
{\it Let $X, Y$ be complete separable metric spaces, $(\Omega,
\mathcal{F}, {P})$ be a probability space. Let measurable mappings $\xi_n:\Omega\to
X,$ $h_n: X\to Y$, $n\geq 0$, be such that
\begin{enumerate}[1)]
\item $\xi_n\to\xi_0, \ n\to\infty,$ in probability $P$;
\item $h_n\to h_0, \ n\to\infty,$ in measure $\nu$, where $\nu$ is a probability
measure on X;
\item for all $n\geq 1$ the distribution $P_{\xi_n}$
of $\xi_n$ is absolutely continuous w.r.t. the measure $\nu$; \item
the sequence of densities $\{\frac{dP_{\xi_n}}{d\nu}: \ n\geq1\}$
is uniformly integrable w.r.t. the measure $\nu$.
\end{enumerate}
Then $h_n(\xi_n)\to h_0(\xi_0),  \ n\to\infty,$ in probability.}
\end{Proposition}
The proof can be found, for example, in \cite{Bogachev07-2}, Corollary 9.9.11 or
\cite{Kulik+00}, Lemma 2.
\begin{proof}[Proof of Lemma \ref{Lemma_Converg_Derivatives_2}]
Let $Z_n(t)$, $n\geq 0$, be a solution of the equation
$$
\left\{
\begin{aligned}
dZ_n(t)&=-Z_n(t)dA_{n,t}(\varphi_n(x)), \ t\in[0,T],\\
Z_n(0)&=E.
\end{aligned}
\right.
$$
where $E$ is the $d$-dimensional identity matrix, $T>0$.
For each $t\in[0,T]$,  $n\geq 0$ the matrix $Z_n(t)$ is invertible, and
$$
\left\{
\begin{aligned}
dZ_n^{-1}(t)&=dA_{n,t}(\varphi_n(x))Z_n^{-1}(t), \ t\in[0,T],\\
Z_n^{-1}(0)&=E,
\end{aligned}
\right.
$$
We get
$$
|Z_n(t)|\leq |E|+\int_0^t|Z_n(s)|d\Var A_{n,s}(\varphi_n(x)).
$$
It follows from Proposition \ref{Prop_Gronwall_Lemma} that
\begin{equation}\label{eq_moment_Z_n}
\sup_{t\in[0,T]}|Z_n(t)|\leq d^{1/2} \exp\left\{\Var A_{n,T}(\varphi_n(x))\right\}.
\end{equation}
Here we use that $|E|=d^{1/2}$.
Similarly,
\begin{equation}\label{eq_moment_Z_1}
\sup_{t\in[0,T]}|Z_n^{-1}(t)|\leq d^{1/2} \exp\left\{\Var A_{n,T}(\varphi_n(x))\right\}.
\end{equation}
Let us prove that
\begin{equation}\label{eq_converg_Z_n_Z_0_Z_n_1_Z_0_1}
\sup_{t\in[0,T]} |Z_n(t)-Z_0(t)|+\sup_{t\in[0,T]} |Z_n^{-1}(t)-Z_0^{-1}(t)|\to 0, \ n\to \infty, \ \mbox{in probability } \mathds{P}.
\end{equation}
We have
\begin{multline*}
|Z_n(t)-Z_0(t)|\leq\\ \left|\int_0^t(Z_0(s)-Z_n(s))dA_{n,s}(\varphi_n(x))\right|+
\left|\int_0^tZ_0(s)\left(dA_{0,s}(\varphi_0(x))-dA_{n,s}(\varphi_n(x))\right)\right|\leq\\
\int_0^t\left|Z_0(s)-Z_n(s)\right|d\Var A_{n,s}(\varphi_n(x))+
\left|\int_0^tZ_0(s)\left(dA_{0,s}(\varphi_0(x))-dA_{n,s}(\varphi_n(x))\right)\right|.
\end{multline*}
By Proposition \ref{Prop_Gronwall_Lemma},
\begin{multline}\label{eq_kkk}
|Z_n(t)-Z_0(t)|\leq\\
\sup_{0\leq u\leq t}\left|\int_0^uZ_0(s)\left(dA_{0,s}(\varphi_0(x))-dA_{n,s}(\varphi_n(x))\right)\right|\exp\{\Var A_{n,t}(\varphi_n(x))\}\leq\\
\sup_{0\leq u\leq t}\left(\left|\int_0^uZ_0(s)\left(dA_{0,s}^+(\varphi_0(x))-dA_{n,s}^+(\varphi_n(x))\right)\right|\right.+\\
\sup_{0\leq u\leq t}\left.\left|\int_0^uZ_0(s)\left(dA_{0,s}^-(\varphi_0(x))-dA_{n,s}^-(\varphi_n(x))\right)\right|\right)\exp\{\Var A_{n,t}(\varphi_n(x))\}.
\end{multline}
Let us apply Proposition \ref{Proposition_Monot_convergence}. Put $h_n(s)=A_{n,s}^+(\varphi_n(x))$, $n\geq 0$, $f(s)=Z_0(s)$. Taking into account  Lemma \ref{Proposition_moments_A} we get that the first summand in the right-hand side of (\ref{eq_kkk}) tends to $0$ as $n\to \infty$ in probability $\mathds{P}$ uniformly in $t\in[0,T]$. The second summand can be treated analogously. Thus we have proved that
\begin{equation*}
\sup_{t\in[0,T]}|Z_n(t)-Z_0(t)|\to 0, \ n\to \infty, \ \mbox{in probability } \mathds{P}.
\end{equation*}
The same relation for $Z_n^{-1}$ can be obtained similarly.

Making use of Ito's formula we get
\begin{multline*}
Z_n(t)Y_{n,t}(x)-Z_0(t)Y_{0,t}(x)=\\
\sum_{k=1}^m
\int_0^t\Big(Z_n(s)\nabla\sigma_{n,k}(s,\varphi_{n,s}(x))Y_{n,s}(x)-Z_0(s)\nabla\sigma_{0,k}(s,\varphi_{0,s}(x))Y_{0,s}(x)\Big)dw_k(s).
\end{multline*}
Applying Ito's formula again, we get for any $K>0$,
\begin{multline}\label{eq_sss}
|Z_n(t)Y_{n,t}(x)-Z_0(t)Y_{0,t}(x)|^2\exp\left\{-K\int_0^t\sum_{k=1}^m\left|\nabla\sigma_{0,k}(s,\varphi_{0,s}(x))\right|^2ds\right\}=\\
\int_0^t\exp\left\{-K\sum_{k=1}^m\int_0^s\left|\nabla\sigma_{0,k}(u,\varphi_{0,u}(x))\right|^2du\right\}\times\\
\sum_{k=1}^m\left(\big|Z_n(s)\nabla\sigma_{n,k}(s,\varphi_{n,s}(x))Y_{n,s}(x)-Z_0(s)\nabla\sigma_{0,k}(s,\varphi_{0,s}(x))Y_{0,s}(x)\big|^2\right.-\\
K|\nabla\sigma_{0,k}(s,\varphi_{0,s}(x))|^2|Z_n(s)Y_{n,s}(x)-Z_0(s)Y_{0,s}(x)|^2\Big)ds+\\
2\int_0^t\exp\left\{-K\int_0^t\sum_{k=1}^m\left|\nabla\sigma_{0,k}(s,\varphi_{0,s}(x))\right|^2ds\right\}\big(Z_n(s)Y_{n,s}(x)-Z_0(s)Y_{0,s}(x)\big)\times\\
\sum_{k=1}^m\Big(Z_n(s)\nabla\sigma_{n,k}(s,\varphi_{n,s}(x))Y_{n,s}(x)-Z_0(s)\nabla\sigma_{0,k}(s,\varphi_{0,s}(x))Y_{0,s}(x)\Big)dw_k(s).
\end{multline}
Taking into account the inequalities \eqref{eq_gnom_2}, (\ref{eq_moment_Z_n}) and Lemma \ref{Lemma_Converg_Derivatives}, one can see that the last summand in the right-hand side of (\ref{eq_sss}) is a square integrable martingale.  The same estimates allow us to write
\begin{multline}\label{eq_ffff}
\E|Z_n(t)Y_{n,t}(x)-Z_0(t)Y_{0,t}(x)|^2\exp\left\{-K\int_0^t\sum_{k=1}^m\left|\nabla\sigma_{0,k}(s,\varphi_{0,s}(x))\right|^2ds\right\}\leq \\I+II,
\end{multline}
where
\begin{multline}
I=\E\int_0^t\exp\left\{-K\sum_{k=1}^m\int_0^s\left|\nabla\sigma_{0,k}(u,\varphi_{0,u}(x))\right|^2du\right\}\times\\
\sum_{k=1}^m\left|Z_n(s)\nabla\sigma_{n,k}(s,\varphi_{n,s}(x))Z_n^{-1}(s)-Z_0(s)\nabla\sigma_{0,k}(s,\varphi_{0,s}(x))Z_0^{-1}\right|^2|Z_n(s)Y_{n,s}(x)|^2ds,\\
II=\E\int_0^t\exp\left\{-K\sum_{k=1}^m\int_0^s\left|\nabla\sigma_{0,k}(u,\varphi_{0,u}(x))\right|^2du\right\}\times\\
\sum_{k=1}^m\Big[\big(|Z_0(s)\nabla\sigma_{0,k}(s,\varphi_{0,s}(x))Z_0^{-1}(s)|^2-
K|\nabla\sigma_{0,k}(s,\varphi_{0,s}(x))|^2\big)\times\\
|Z_n(s)Y_{n,s}(x)-Z_0(s)Y_{0,s}(x)|^2\Big]ds.
\end{multline}

It follows from the estimates (\ref{eq_moment_Z_n}), (\ref{eq_moment_Z_1}) that for large enough $K$, $II\leq0$.

Consider $I$. First using Proposition \ref{Proposition_Kulik} let us show that for $1\leq k\leq m$, $s\geq 0$, and $x\in\R$, $\nabla\sigma_{n,k}(s,\varphi_{n,s}(x))\to
\nabla\sigma_{0,k}(s,\varphi_{0,s}(x))$, $n\to\infty$, in probability. Fix $s\geq 0$,  $x\in\R$, and $1\leq k\leq m$. In the conditions of Proposition \ref{Proposition_Kulik} we put  $\xi_n=\varphi_{n,s}(x)$, $n\geq0$. Lemma \ref{Lemma_Converg_Solutions} entails the convergence $\xi_n\to\xi_0$, $n\to\infty$, in probability.
Put $X=\R$, $Y=\R\times\R$, $\nu(dx)= C\frac{dx}{1+|x|^{d+1}}$, where $C$ is a constant such that $\nu$ is a probability measure on $\R$.
For fixed $s,x$, and $k$ put $h_n=\nabla\sigma_{n,k}(s,\cdot)$, $h_0=\nabla\sigma_{0,k}(s,\cdot)$. Since for each $s\in [0,T]$, $\nabla\sigma_{n,k}(s,\cdot)\to\nabla\sigma_{0,k}(s,\cdot)$, $n\to\infty$, in $L_2(\R)$ we can
suppose without lost of generality that  $\nabla\sigma_{n,k}(s,y)\to\nabla\sigma_{0,k}(s,y)$, $n\to\infty$, for  each $1\leq k\leq m$ and almost all $s\in[0,T]$, $y\in\R$, with respect to the Lebesgue measure. Then for almost all $s\in[0,T]$, $h_n\to h_0$, $n\to\infty$, in the measure $\nu$.
Notice that the processes $(\varphi_{n,t}(x))_{t\geq0}$, $n\geq0$, possess  transition probability densities. Thus the distributions $P_{\xi_n}$, $n\geq0$, are absolutely continuous w.r.t. the Lebesgue measure on $\R$ and, consequently, w.r.t. the measure $\nu$.  Making use of the estimates \eqref{eq_gaussian_estimates} it is easily seen that the sequence of densities $\left\{\frac{dP_{\xi_n}}{d\nu}:n\geq 1\right\}$ is uniformly integrable w.r.t. the measure $\nu$. Therefore, all the assumptions of Proposition \ref{Proposition_Kulik} are fulfilled, and for almost all $s\in[0,T]$, and all $x\in\R$,
\begin{equation}\label{eq_nabla_sigma_converg}
\nabla\sigma_{n,k}(s,\varphi_{n,s}(x))\to
\nabla\sigma_{0,k}(s,\varphi_{0,s}(x)), \ n\to\infty, \ \mbox{in probability } \mathds{P}.
\end{equation}

Let us return to $I$. We have
\begin{multline*}
\left|Z_n(s)\nabla\sigma_{n,k}(s,\varphi_{n,s}(x))Z_n^{-1}(s)-Z_0(s)\nabla\sigma_{0,k}(s,\varphi_{0,s}(x))Z_0^{-1}(s)\right|\leq\\
|Z_n(s)\nabla\sigma_{n,k}(s,\varphi_{n,s}(x))||Z_n^{-1}(s)-Z_0^{-1}(s)|+\\
|Z_n(s)||\nabla\sigma_{n,k}(s,\varphi_{n,s}(x))-\nabla\sigma_{0,k}(s,\varphi_{0,s}(x))||Z_0^{-1}(s)|+\\
|Z_n(s)-Z_0(s)||\nabla\sigma_{0,k}(s,\varphi_{0,s}(x))Z_0^{-1}(s)|.
\end{multline*}
Making use the H\"older inequality as it was done in (\ref{eq_Holder_h}), taking into account the estimates (\ref{eq_moment_Z_n}), (\ref{eq_moment_Z_1}) and the relations (\ref{eq_converg_Z_n_Z_0_Z_n_1_Z_0_1}), (\ref{eq_nabla_sigma_converg}), we get that the first expectation in the right-hand side of (\ref{eq_ffff}) tends to $0$ as $n\to\infty$. Thus we obtained that
\begin{equation*}
\sup_{t\in[0,T]}|Z_n(t)Y_{n,t}(x)-Z_0(t)Y_{0,t}(x)|\to 0, \ n\to\infty, \ \mbox{in probability } \mathds{P}.
\end{equation*}
Now the assertion of the lemma can be deduce from the inequality
$$
|Y_{n,t}(x)-Y_{0,t}(x)|\leq |Z_n^{-1}(t)||Z_n(t)Y_{n,t}(x)-Z_0(t)Y_0(t)|+
|Z_n^{-1}(t)-Z_0^{-1}(t)||Z_0(t)Y_0(t)|
$$
using standard arguments for the proof of uniform convergence. Lemma \ref{Lemma_Converg_Derivatives_2} is proved.
\end{proof}
Making use of
Lemma \ref{Lemma_Converg_Solutions} and the dominated convergence
theorem, for each $T>0$, $p\geq 1$, we get the relation
$$
\mathds{E}\sup_{0\leq t\leq T}\int_U|\varphi_{n,t}(x)-\varphi_t(x)|^pdx\to 0 , \
n\to \infty,
$$
valid for any bounded domain $U\subset\R$.
Then there exists a
subsequence $\{n_k: \ k\geq 1\}$ such that
$$
\sup_{0\leq t\leq T}\int_U|\varphi_{n_k,t}(x)-\varphi_{t}(x)|^pdx\to 0 \
\mbox{a.s. as} \ k\to \infty.
$$
Without loss of generality we can suppose that
\begin{equation}\label{eq_convergence_solutions}
\sup_{0\leq t\leq T}\int_U|\varphi_{n,t}(x)-\varphi_{t}(x)|^pdx\to 0 \ \mbox{a.s.
as}\ n\to \infty.
\end{equation}

It follows from Lemma
\ref{Lemma_Converg_Derivatives_2} in the similar way that for each $T>0$, $p\geq0$,
\begin{equation}\label{eq_convergence_derivatives}
\sup_{0\leq t\leq T}\int_U|Y_{n,t}(x)-Y_{t}(x)|^pdx\to 0, \ n\to\infty, \
\mbox{almost surely}.
\end{equation}

Since the Sobolev space is a Banach space,  the relations
(\ref{eq_convergence_solutions}),
(\ref{eq_convergence_derivatives}) mean that $Y_t(x)$ is the
 matrix
of the Sobolev derivatives of the solution to (\ref{eq_main}) and \eqref{eq_Sobolev_derivative} holds.
\subsection{} We consider the general case making use of localization.  Let the coefficients of equation (\ref{eq_main}) satisfy the assumptions of Theorem \ref{Theorem_main}.
Let the functions $\beta, \gamma \in C^1(\R)$ be such that $|\beta(x)|\leq 1$; $\beta(x)=1$, if $|x|\leq 2$; $\beta(x)=0$, if $|x|>3$; $|\gamma(x)|\leq 1$; $\gamma(x)=0$, if $|x|\leq 1$; $\gamma(x)=1$, if $|x|>3/2$.  For $R>1$, put $\beta_R(x)=\beta({x}/{R})$, $\gamma_R(x)=\gamma({x}/{R})$. Consider the SDE
\begin{equation}\label{eq_axilary}
\left\{
\begin{aligned}
d\vp_{R,t}(x)&=a(t,\vp_{R,t}(x))\beta_R(\vp_{R,t}(x))dt+\\ &\sum_{k=1}^{m}\sigma_{k}(t,\vp_{R,t}(x))\beta_R(\vp_{R,t}(x))dw_{k}(t)+
\sum_{j=1}^{m}\widetilde{\sigma}_{j}\gamma_R(\vp_{R,t}(x))d\widetilde w_{j}(t),\\
\vp_{R,0}(x)&=x,
\end{aligned}\right.
\end{equation}
where $\widetilde\sigma$ is a $d\times m$ constant matrix such that $\widetilde\sigma\widetilde\sigma^\ast>0$; $(\widetilde w(t))_{t\geq0}=(\widetilde w_1(t),\dots,\widetilde w_m(t))_{t\geq0}$ is an $m$-dimensional Wiener process independent of  $(w(t))_{t\geq 0}$.

Similarly to Lemma \ref{Lemma_Converg_Solutions}, for each $x\in\R$, we get
$$
\sup_{R>1}\E\big(|\vp_{R,t}(x)|^p+|\vp_t(x)|^p\big)<\infty.
$$
Note that $\varphi_{R,t}(x)$ coincides with $\varphi_{t}(x)$ for $t\leq \tau_R$, where $\tau_R=\inf\{s\geq0: \varphi_s(x)\geq R\}$.
Then from the boundedness of the coefficients of (\ref{eq_main})  we obtain that for all $x\in\R$,
$$
\P\{\sup_{0\leq t\leq T}|\varphi_{R,t}(x)-\varphi_t(x)|>\varepsilon\}\leq \P\{\sup_{0\leq t\leq T}|\varphi_t(x)|>R\}\to 0, \ R\to\infty.
$$
It is not difficult,   by analogy to (\ref{eq_convergence_solutions}), to arrive at the relation
\begin{equation}\label{eq_jjjj}
\sup_{0\leq t\leq T}\int_U|\varphi_{R_k,t}(x)-\varphi_{t}(x)|^pdx\to 0 \ \mbox{almost surely
as} \ k\to\infty
\end{equation}
valid for all $x\in\R$, $p\geq 1$, and a sequence $\{R_k:k\geq 1\}$ such that $R_k\to\infty$, $k\to\infty$.
It follows from Lemma \ref{Lemma_Converg_Derivatives} that for all $x\in\R$,
\begin{equation}\label{eq_finiteness_Y_R}
\sup_{R>1}\E\big(\sup_{0\leq t\leq T}(|Y_{R,t}(x)|^p+|Y_t(x)|^p)\big)<\infty.
\end{equation}
According to subsection 3.1, for each $k\geq 1$ there exists the derivative $\nabla\vp_{R_k,t}(x)$ which, for almost all $x\in\R$, is equal to the solution of the equation
\begin{multline}\label{eq_derivative_3}
Y_{R_k,t}(x)=E+\int_0^t \beta_{R_k}(\vp_{R_k,s}(x))dA_{R_k,s}(\varphi_{R_k}(x))Y_{R_k,s}(x)+\\
\int_0^t \nabla\beta_{R_k}(\vp_{R_k,s}(x))a(s,\vp_{R_k,s}(x))Y_{R_k,s}(x)ds+\\
\sum_{k=1}^{m}\int_0^t\nabla\sigma_{k}(s,\vp_{R_k,s}(x))\beta_{R_k}(\vp_{R_k,s}(x))Y_{R_k,s}(x)dw_{k}(s)+\\
\sum_{k=1}^{m}\int_0^t\sigma_{k}(s,\vp_{R_k,s}(x))\nabla\beta_{R_k}(\vp_{R_k,s}(x))Y_{R_k,s}(x)dw_{k}(s)+\\
\sum_{j=1}^{m}\int_0^t\widetilde{\sigma}_{j}\nabla\gamma_{R_k}(\vp_{R_k,s}(x))d\widetilde w_{j}(s).
\end{multline}
Note that $A_{R_k,t}(\vp_{R_k}(x))=A_{t}(\vp(x))$, for $t\leq \tau_{R_k}$, where $\tau_{R_k}=\inf\{t:\vp_{t}(x)\geq R_k\}$. Therefore, equation (\ref{eq_derivative_3}) coincides with equation (\ref{eq_derivative_main}) for $t\leq \tau_{R_k}$, $k\geq 1$.
As $\tau_{R_k}\to\infty$, $k\to\infty$, we deduce that
\begin{equation}\label{eq_jjjjj}
\sup_{0\leq t\leq T}\int_U|Y_{R_k,t}(x)-Y_{t}(x)|^pdx\to 0, \ k\to\infty, \
\mbox{almost surely},
\end{equation}
for $T>0$, any bounded domain $U\subset\R$, and a sequence $\{R_k: k\geq 1\}$ such that $R_k\to \infty$ as $k\to\infty$.
From (\ref{eq_jjjj}) and (\ref{eq_jjjjj}) we get that $Y_t(x)=\nabla\vp_t(x)$, $t\geq 0$, for $\lambda$-a.a. $x\in\R$, almost surely.

Let us verify \eqref{eq_derivative_main2}. Given $R>1$, the coefficients of equation \eqref{eq_axilary} satisfy all the localizing conditions imposed on the coefficients  of equation \eqref{eq_main} in Subsection \ref{subsction_compact_case}. Denote by $\varphi_{n,t}^R$, $n\geq 1$,  a solution to equation of the form \eqref{eq_axilary} with smooth coefficients such that for $p\geq1$, $T>0$, and $x\in\R$,
\begin{equation}\label{eq_star1}
\E\left(\sup_{0\leq t\leq T}|\varphi_{n,t}^R(x)-\varphi_{R,t}(x)|^p\right)\to 0, \ n\to\infty,
\end{equation}
\begin{equation}\label{eq_star2}
\E\left(\sup_{0\leq t\leq T}|Y_{n,t}^R(x)-Y_{R,t}(x)|^p\right)\to 0, \ n\to\infty.
\end{equation}
Then  for all $x,h \in \R$, $v\in \mathds R$,
$$
\vf_{n,t}^R(x+v h)=\vf_{n,t}^R(x)+h\int_0^v Y_{n,t}^R(x+uh)du.
$$
This equation, \eqref{eq_star1},  \eqref{eq_star2}, and Lemma \ref{Lemma_Converg_Derivatives} imply that
for all $x,h\in\R, v\in\mathds{R}$, and $R>1$,
$$
\vf_{R,t}(x+v h)=\vf_{R,t}(x)+h\int_0^v Y_{R,t}(x+uh)du.
$$

By  \eqref{eq_jjjj}, \eqref{eq_finiteness_Y_R}, and \eqref{eq_jjjjj} we get the equality
\begin{equation}\label{eq_NL}
\vf_{t}(x+v h)=\vf_{t}(x)+h\int_0^v Y_{t}(x+uh)du
\end{equation}
valid for all $x,h\in\R, v\in\mathds{R}$, and $R>1$,
To obtain \eqref{eq_derivative_main2} it remains to prove the $L_p$-continuity of $Y_t(x)$  w.r.t. $x$. Note that Lemma \ref{Lemma_converg_W_functionals} implies the convergence
\begin{equation*}
A_t(\vp(x))\to A_t(\vp(x_0)), \ x\to x_0, \ \mbox{in probabillty}.
\end{equation*}
Then
\begin{equation}\label{eq_rrr}
Y_t(x)\to Y_t(x_0), \ x\to x_0, \ \mbox{in probability}.
\end{equation}
This together with Lemma \ref{Lemma_Converg_Derivatives} entails convergence in $L_p$, $p>0$.
Now \eqref{eq_derivative_main2} follows from \eqref{eq_NL} and \eqref{eq_rrr}.
This completes the proof of Theorem \ref{Theorem_main}.

\section{Appendix. Convergence of transition probability densities}\label{Section_Conv_Densities}
In this section we prove the convergence of the transition probability densities of the processes $(\varphi_{n,t})_{t\geq0}$, $n\geq1$,   to that of the process  $(\varphi_t)_{t\geq0}$ (Lemma \ref{Lemma_Converg_Densities}, see below), which entails the convergence of characteristics of W-functionals (Lemma  \ref{Lemma_Converg_Character}, see below).   The latter result  is the basis of  the proof of Lemma \ref{Lemma_converg_W_functionals}. We make use of the parametrix method considering the transition probability densities of the processes with $a_n\equiv0$, $n\geq1$, as the initial ones.

Suppose that $\sigma$ satisfies the conditions of Theorem \ref{Theorem_main} and $\sigma(t,x)=\widetilde{\sigma}=const$ for $t\geq0$,  $x\in\R$ such that $|x|\geq R$, $\widetilde{\sigma}\widetilde{\sigma}^\ast>0$. Let $\sigma_n, n\geq1,$ be defined by equation \eqref{eq_a_n}. Then $\sigma_n\to \sigma, n\to \infty,$ uniformly in $(t,x)\in [0,T]\times\R$. Recall that we can assume that $\sigma_n(t,x)=\widetilde{\sigma}$ for all $n\geq 1$, $t\geq0$,  and $x\in\R$ such that $|x|\geq R$.

Denote  $\sigma_0=\sigma$, $\varphi_0=\varphi$, and for $n\geq 0$ put
$$
b_n=\sigma_n\sigma^\ast_n.
$$
Then  $b_n\to b_0$, $n\to\infty,$ uniformly in $(t,x)\in[0,T]\times\R$, $T>0$.

Consider the parabolic equation
$$
\frac{\partial u_n(s,x)}{\partial s}+\frac12\sum_{i,j=1}^db_n^{ij}(s,x)\frac{\prt^2 u_n(s,x)}{\prt x_i\prt x_j}=0, \ n\geq 0.
$$
It is well known that the H\"older continuity and uniform ellipticity of $b_n$ provide the existence of a fundamental solution (e.g., \cite{Ladyzhenskaya+67}, Ch. IV, \S 11), which we denote by $g_n(s,x,t,y)$ (recall that now $a_n\equiv 0$). The function $g_n(s,x,t,y)$, $0\leq s<t\leq T, \ x\in\R, \ y\in\R$, is the transition probability density of the diffusion process which is a solution of the SDE
$$
x_n(t)=x_n(s)+\sum_{k=1}^m\int_s^t\sigma_{n,k}(u,x_n(u))dw_k(u).
$$
By \cite{Portenko95}, Ch.II, Lemma 3,
\begin{eqnarray}
g_n(s,x,t,y)&\to &g_0(s,x,t,y), \ n\to\infty, \label{eq_g_n_to_g_0}\\
\frac{\prt g_n(s,x,t,y)}{\prt x_i}&\to&\frac{\prt g_0(s,x,t,y)}{\prt x_i}, \ 1\leq i\leq d, \ n\to\infty,\label{eq_g_n_to_g_01}
\end{eqnarray}
 uniformly in every domain
$$
\mathcal{D}_{\delta}^T=\{(s,x,t,y):0\leq s<t\leq T, x\in\R, y\in\R, t-s+|x-y|\geq\delta\},
$$
for any fixed $\delta>0, T>0$.

Furthermore, for $0\leq s<t\leq T$, $x\in\R$, $y\in\R$, the estimates
\begin{equation}\label{eq_g_n_estimate}
|\nabla_x^l g_n(s,x,t,y)|\leq C(t-s)^{-\frac{d+l}{2}}\exp\left\{-c\frac{|y-x|^2}{t-s}\right\}
\end{equation}
hold true. Here $n\geq 0$, $l=0,1,2$, $C,c$ are positive constants which depend only on $d, T$ and $\|b_0\|_{T,\infty}$.

Now let $a$ satisfy the condition of Theorem \ref{Theorem_main}, and $a(t,x)=0$ for $t\geq 0, |x|>R$. Put $a_0=a$, and $\varphi_{0,t}(x)=\varphi_{t}(x)$, $t\geq 0, x\in\R$, where $\varphi_t(x)$ is a solution of equation \eqref{eq_main}. Let for $n\geq 1$, $a_n$ be defined by \eqref{eq_a_n}, and $\varphi_{n,t}(x)$ be a solution of equation \eqref{eq_main_n}.
Denote by $G_n(s,x,t,y), \ n\geq0$, the transition probability density of the process $(\varphi_{n,t})_{t\geq0}$. Then $G_n(s,x,t,y)$ can be constructed by the perturbation method (see \cite{Portenko90}, Ch. 2) as a solution of the integral equation:
\begin{equation}\label{eq_G_n}
G_n(s,x,t,y)=g_n(s,x,t,y)+\int_s^td\tau\int_{\R}g_n(s,x,\tau,z)\left(\nabla_z G_n(\tau,z,t,y),a_n(\tau,z)\right) dz,
\end{equation}
which satisfies the estimate
\begin{equation}\label{eq_estimate_G_n}
|\nabla_x^l G_n(s,x,t,y)|\leq C^\prime(t-s)^{-\frac{d+l}{2}}\exp\left\{-c^\prime\frac{|y-x|^2}{t-s}\right\}
\end{equation}
in any domain $0\leq s<t\leq T$, $x\in\R$, $y\in\R$, for $n\geq 0, \ l=0,1$. The constants $C^\prime$, $c^\prime$ can be chosen uniformly in $n$.

It follows from \cite{Portenko90}, Theorem 2.1, that for $n\geq1$, the function $G_n(s,x,t,y)$ is a fundamental solution of the parabolic equation
$$
\frac{\partial u_n(s,x)}{\partial s}+\frac12\sum_{i,j=1}^db_n^{ij}(s,x)\frac{\prt^2 u_n(s,x)}{\prt x_i\prt x_j}+\sum_{i=1}^d a^i_n(s,x)\frac{\prt u_n(s,x)}{\prt x_i}=0.
$$
\begin{Remark}\label{Remark_uniform_cont}
Following the construction of $G_0(s,x,t,y)$ in \cite{Portenko90} one can observe that $G_0(s,x,t,y)$ is uniformly continuous in $y$ uniformly on $|t-s|>\delta$, $x\in\R$, $\delta>0$.
\end{Remark}

\begin{Lemma}\label{Lemma_Converg_Densities}
$G_n(s,x,t,y)\to G_0(s,x,t,y)$, $n\to\infty$, uniformly on $\mathcal{D}_{\delta}^T$
for any fixed $\delta>0, T>0$.
\end{Lemma}
\begin{proof}
We use the idea of the proof from \cite{Portenko90}, Lemma 2.6.

Denote
\begin{equation}\label{eq_U_n_1}
U_n(s,x,t,y)=\nabla_x G_n(s,x,t,y)-\nabla_x G_0(s,x,t,y).
\end{equation}
Equation (\ref{eq_G_n}) entails the relation
\begin{multline*}
U_n(s,x,t,y)=\nabla_xg_n(s,x,t,y)-\nabla_xg_0(s,x,t,y)+\\
\int_s^td\tau\int_{\R}\nabla_xg_n(s,x,\tau,z)\left(\nabla_zG_n(\tau,z,t,y),a_n(\tau,z)\right) dz-\\
\int_s^td\tau\int_{\R}\nabla_xg_0(s,x,\tau,z)\left(\nabla_zG_0(\tau,z,t,y),a_0(\tau,z)\right) dz=\\
\nabla_xg_n(s,x,t,y)-\nabla_xg_0(s,x,t,y)+\\
\int_s^td\tau\int_{\R}\nabla_xg_0(s,x,\tau,z)\left(U_n(\tau,z,t,y),a_n(\tau,z)\right) dz+\\
\int_s^td\tau\int_{\R}(\nabla_xg_n(s,x,\tau,z)-\nabla_xg_0(s,x,\tau,z))\left(\nabla_zG_n(\tau,z,t,y),a_n(\tau,z)\right) dz+\\
\int_s^t d\tau\int_{\R}\nabla_xg_0(s,x,\tau,z)\left(\nabla_zG_0(\tau,z,t,y),a_n(\tau,z)-a_0(\tau,z)\right) dz.
\end{multline*}
So
\begin{equation}\label{eq_U_n}
U_n(s,x,t,y)=\mathcal{A}_nU_n(s,x,t,y)+r_n(s,x,t,y),
\end{equation}
where
$$
\mathcal{A}_nU_n(s,x,t,y)=\int_s^td\tau\int_{\R}\nabla_xg_0(s,x,\tau,z)\left(U_n(\tau,z,t,y),a_n(\tau,z)\right) dz,
$$
$$
r_n(s,x,t,y)=\sum_{k=1}^3 I_n^k(s,x,t,y),
$$
\begin{eqnarray*}
I_n^1(s,x,t,y)=&\nabla_xg_n(s,x,t,y)-\nabla_xg_0(s,x,t,y),\\
I_n^2(s,x,t,y)=&\int_s^td\tau\int_{\R}(\nabla_xg_n(s,x,\tau,z)-\nabla_xg_0(s,x,\tau,z))\left(\nabla_zG_n(\tau,z,t,y),a_n(\tau,z)\right) dz,\\
I_n^3(s,x,t,y)=&\int_s^t d\tau\int_{\R}\nabla_xg_0(s,x,\tau,z)\left(\nabla_zG_0(\tau,z,t,y),a_n(\tau,z)-a_0(\tau,z)\right) dz.
\end{eqnarray*}
Recall that $a_n(t,x), \ n\geq0$, are bounded measurable and they have  compact supports in $x$. So $a_n\in L_p([0,T]\times\R)$ for all $T>0, p>0, n\geq0$. Fix $p>d+2$. Making use of the H\"older inequality and the estimate \eqref{eq_estimate_G_n} we have
\begin{multline}\label{eq_I_n_2}
I_n^2(s,x,t,y)\leq\\
\int_s^td\tau\int_{\R}|\nabla_xg_n(s,x,\tau,z)-\nabla_xg_0(s,x,\tau,z)|\left|\nabla_zG_n(\tau,z,t,y)\right| |a_n(\tau,z)| dz\leq\\
K\left(\int_s^td\tau\int_{\R}|\nabla_xg_n(s,x,\tau,z)-\nabla_xg_0(s,x,\tau,z)|^q\times\right.\\
\left.(t-\tau)^{-\frac{d+1}{2}q}\exp\left\{-cq\frac{|y-z|^2}{t-\tau}\right\}dz\right)^{1/q}
\left(\int_s^td\tau\int_{\R}|a_n(\tau,z)|^p dz\right)^{1/p},
\end{multline}
where $K,c$ are positive constants, $1/p+1/q=1$. It follows from (\ref{eq_g_n_to_g_01}) and \cite{Portenko95}, Ch. II, Lemma 2 that $I_n^2(s,x,t,y)\to 0, \ n\to\infty$, uniformly on $\mathcal D_\delta^T$ for any $\delta>0, \ T>0$.
The  relation (\ref{eq_g_n_to_g_01}) gives also that $I_n^1(s,x,t,y)\to 0$, $n\to\infty$, uniformly on $\mathcal D_\delta^T$. Consider $I_n^3(s,x,t,y)$. We have
\begin{multline}\label{eq_III}
I_n^3(s,x,t,y)\leq\int_s^td\tau\int_{\R}|\nabla_xg_0(s,x,\tau,z)||\nabla_zG_0(\tau,z,t,y)||a_n(\tau,z)-a_0(\tau,z)| dz\leq\\
K\left(\int_s^td\tau\int_{\R}|a_n(\tau,z)-a_0(\tau,z)|^p dz\right)^{1/p}\times\\
\left(\int_s^td\tau\int_{\R}(\tau-s)^{-\frac{d+1}{2}q}\exp\left\{-cq\frac{|z-x|^2}{\tau-s}\right\}(t-\tau)^{-\frac{d+1}{2}q}\exp\left\{-cq\frac{|y-z|^2}{t-\tau}\right\} dz\right)^{1/q}=\\
K^\prime\|a_n-a_0\|_{p,T}(t-s)^{-\frac{d+1}{2}+\gamma}\exp\left\{-c\frac{|y-x|^2}{t-s}\right\},
\end{multline}
where $K^\prime$ is a constant, $\gamma=\frac{p-d-2}{2p}$, $p>d+2$, $\|a\|_{p,T}=\|a\|_{L_p([0,T]\times\R)}$. For the proof of the last equality in \eqref{eq_III} see, e.g.,
\cite{Friedman64}, Ch.1, \S 4, Lemma 3. Then $I_n^3(s,x,t,y)\to 0$, $n\to\infty$, uniformly on $\mathcal D_\delta^T$. Thus we conclude that
$$
r_n(s,x,t,y)\to 0, \ n\to\infty, \ \mbox{uniformly on} \ \mathcal{D}_{\delta}^T \ \mbox{for any } \delta>0, \ T>0.
$$
Moreover, from (\ref{eq_g_n_estimate}), (\ref{eq_I_n_2}), and  (\ref{eq_III}) we obtain the following estimate
\begin{equation}\label{eq_estimate_r_n}
|r_n(s,x,t,y)|\leq H(t-s)^{-\frac{d+1}{2}}\exp\left\{-c\frac{|y-x|^2}{t-s}\right\}
\end{equation}
valid in every domain of the form $0\leq s<t\leq T$, $x,y\in\R$. Here $H$ is a positive constant. We obtain the above inequality for $I_n^2$ in the way similar to that for $I_n^3$.

By (\ref{eq_estimate_G_n}) and  \eqref{eq_U_n_1} for all $0\leq s<t\leq T$, $x,y\in\R$,
\begin{equation}
|U_n(s,x,t,y)|\leq H^\prime(t-s)^{-\frac{d+1}{2}}\exp\left\{-c\frac{|y-x|^2}{t-s}\right\},
\end{equation}
where $H^\prime$ is a positive constant.
Denote by $\mathcal A_n^k$ is the $k$-th power of the operator $\mathcal A_n$. Repeating the argument of (\ref{eq_III}), we get
$$
|\mathcal A_n^kU_n(s,x,t,y)|\leq C_k\|a\|^k_{p,T}(t-s)^{-\frac{d+1}{2}+k\gamma}\exp\left\{-c\frac{|y-x|^2}{t-s}\right\},
$$
where
$$
C_k=H^\prime C^k\left(\frac{\pi}{cq}\right)^{\frac{kd}{2q}}\left(\frac{\Gamma(\beta)}{\Gamma((k+1)\beta)}\right)^{1/q}, \ q=\frac{p}{p-1}, \ \gamma=\frac{p-d-2}{2p}, \ \beta=q\gamma.
$$
Here $k=0,1,2,\dots$, $0\leq s<t\leq T$, $x,y\in\R$,  $C$ is a constant from the inequality (\ref{eq_g_n_estimate}).
It follows from these estimates that
\begin{equation}\label{eq_lim_AnUn}
\lim_{k\to\infty}\sup_n\sup_{0\leq s<t\leq T, \ x,y\in\R}|\mathcal A_n^k U_n(s,x,t,y)|=0.
\end{equation}
Using the estimate \eqref{eq_estimate_r_n} and arguing similarly we get for $k=0,1,2,\dots$, that
\begin{equation}\label{eq_estim_r_n}
|\mathcal A_n^kr_n(s,x,t,y)|\leq C_k^\prime\|a\|^k_{p,T}(t-s)^{-\frac{d+1}{2}+k\gamma}\exp\left\{-c\frac{|y-x|^2}{t-s}\right\},
\end{equation}
where
$$
C_k^\prime=H^\prime C^k\left(\frac{\pi}{cq}\right)^{\frac{kd}{2q}}\left(\frac{\Gamma(\beta)}{\Gamma((k+1)\beta)}\right)^{1/q}.
$$
Iterating the relation (\ref{eq_U_n}) and taking into account (\ref{eq_lim_AnUn}) we deduce that
\begin{equation}\label{eq_series_U_n}
U_n(s,x,t,y)=\sum_{k=0}^\infty\mathcal A_n^kr_n(s,x,t,y).
\end{equation}
The estimates (\ref{eq_estim_r_n}) provide the convergence of the series in the right-hand side of (\ref{eq_series_U_n}) uniformly in $n$ on $\mathcal{D}_\delta^T$. To prove that
\begin{equation}\label{eq_U_n_to_0}
\lim_{n\to\infty}U_n(s,x,t,y)=0
\end{equation}
on $\mathcal D_\delta^T$ it is enough to show that $\mathcal A_n^kr_n(s,x,t,y)\to 0$, $n\to\infty$, for every fixed $k=0,1,2,\dots$. This can be easily obtained by induction.

For the difference $G_n-G_0$ from \eqref{eq_G_n} we have
\begin{equation*}
G_n(s,x,t,y)-G_0(s,x,t,y)=\sum_{k=1}^4 H_n^k(s,x,t,y),
\end{equation*}
where
\begin{eqnarray*}
H_n^1(s,x,t,y)&=&g_n(s,x,t,y)-g_0(s,x,t,y),\\
H_n^2(s,x,t,y)&=&\int_s^td\tau\int_{\R}(g_n(s,x,\tau,z)-g_0(s,x,\tau,z))\left(\nabla G_n(\tau,z,t,y),a_n(\tau,z)\right) dz,\\
H_n^3(s,x,t,y)&=&\int_s^td\tau\int_{\R}g_0(s,x,\tau,z)\left(\nabla G_n(\tau,z,t,y)-\nabla G_0(\tau,z,t,y),a_n(\tau,z)\right) dz,\\
H_n^4(s,x,t,y)&=&\int_s^td\tau\int_{\R}g_0(s,x,\tau,z)\nabla G_0(\tau,z,t,y)(a_n(\tau,z)-a_0(\tau,z)) dz.
\end{eqnarray*}
By (\ref{eq_g_n_to_g_0}) we have $H_n^1(s,x,t,y)\to 0, n\to\infty$, uniformly on $\mathcal D_\delta^T$ for any $\delta>0, T>0$. It follows form (\ref{eq_g_n_to_g_0}), (\ref{eq_U_n_to_0}), and the dominated convergence theorem  that $H_n^2(s,x,t,y)\to 0, n\to\infty$, and $H_n^3(s,x,t,y)\to 0, n\to\infty$, uniformly on $\mathcal D_\delta^T$. Finally, $H_n^4$ satisfies the inequality
$$
|H_n^4(s,x,t,y)|\leq K\|a_n-a\|_{p,T}(t-s)^{-\frac d2+\gamma}\exp\left\{-c\frac{|y-x|^2}{t-s}\right\}.
$$
This implies that $H_n^4(s,x,t,y)\to 0$ as $n\to\infty$ uniformly on $\mathcal D_\delta^T$.
The  lemma is proved.
\end{proof}
\begin{Lemma}\label{Lemma_Converg_Character}
Let $\widetilde\nu(dt,dy)=\nu(t,dy)dt$ be a measure of the class $\mathcal{K}$ such that $\supp(\widetilde\nu)\subset[0,T]\times U$ for some $T>0$ and compact set $U\in\R$.
Then
$$
\int_{t_0}^{t_0+t}ds\int_{\R}G_n(t_0,x,s,y)(\nu\ast \omega_n)(s,dy)\to\int_{t_0}^{t_0+t}ds\int_{\R}G_0(t_0,x,s,y)\nu(s,dy), \ n\to\infty,
$$
uniformly on $0\leq t_0<t_0+t\leq T, x\in\R$.
\end{Lemma}
\begin{proof}
We can write
\begin{multline*}
\left|\int_{t_0}^{t_0+t}ds\int_{\R}G_n(t_0,x,s,y)(\nu\ast \omega_n)(s,dy)-\int_{t_0}^{t_0+t}ds\int_{\R}G_0(t_0,x,s,y)\nu(s,dy)\right|\\
\leq I_n^1(t_0,t,x)+I_n^2(t_0,t,x),
\end{multline*}
where
\begin{eqnarray*}
I_n^1(t_0,t,x)&=&\left|\int_{t_0}^{t_0+t}ds\int_{\R}\Big(\big(G_n(t_0,x,s,\cdot)\ast \omega_n\big)(y)-\big(G_0(t_0,x,s,\cdot)\ast \omega_n\big)(y)\Big)\nu(s,dy)\right|,\\
I_n^2(t_0,t,x)&=&\left|\int_{t_0}^{t_0+t}ds\int_{\R}\Big(\big(G_0(t_0,x,s,\cdot)\ast \omega_n\big)(y)-G_0(t_0,x,s,\cdot)(y)\Big)\nu(s,dy)\right|.
\end{eqnarray*}
For any $\delta>0$ we have
\begin{multline}\label{eq_llll}
I_n^1(t_0,t,x)\leq\int_{t_0}^{t_0+\delta}ds\int_{\R}\big(G_n(t_0,x,s,\cdot)\ast\omega_n\big)(y)\nu(s,dy)+\\
\int_{t_0}^{t_0+\delta}ds\int_{\R}\big(G_0(t_0,x,s,\cdot)\ast\omega_n\big)(y)\nu(s,dy)+\\
\int_{t_0+\delta}^{t_0+t}ds\int_{\R}\big|G_n(t_0,x,s,y)-G_0(t_0,x,s,y)\big|\nu(s,dy).
\end{multline}
The estimates \eqref{eq_gaussian_estimates} entail
\begin{multline}\label{eq_l123}
\int_{t_0}^{t_0+\delta}ds\int_{\R}\big(G_n(t_0,x,s,\cdot)\ast\omega_n\big)(y)\nu(s,dy)\leq\\
\int_{t_0}^{t_0+\delta}ds\int_{\R}\nu(s,dy)\int_{\R}G_n(t_0,x,s,y-z)\omega_n(z)dz\leq\\
C\int_{t_0}^{t_0+\delta}ds\int_{\R}\nu(s,dy)\int_{\R}\exp\left\{-c\frac{|y-(z+x)|^2}{s-t_0}\right\}\omega_n(z)dz\leq\\
C\sup_{\widetilde x\in\R}\int_{t_0}^{t_0+\delta}ds\int_{\R}\exp\left\{-c\frac{|y-\widetilde x|^2}{s-t_0}\right\}\nu(s,dy).
\end{multline}
Because of the condition (\ref{Cond_A_prime}), for each $\varepsilon>0$, we can choose $\delta$ so small that for all $t_0\in[0,T-\delta]$ the right-hand side of (\ref{eq_l123}) does not exceed $\varepsilon/2$.
The same estimate for the second summand in the right-hand side of \eqref{eq_llll} can be obtained similarly.

To prove the convergence of the last item in the right-hand side of \eqref{eq_llll} to zero we note that for each $T>0$ and compact set $U\subset\R$ there exists $C>0$ such that
\begin{equation}\label{eq_l234}
\sup_{t_0\in[0,\infty)}\int_{t_0+\delta}^{t_0+T}ds\int_U\nu(s,dy)<C.
\end{equation}
Indeed, let  $R>0$ be such that $U\subset B(0,R)$. We have that for all  $s\in[t_0+\delta,t_0+T]$, $x\in\R$, and $y\in \R$ such that $|y|\leq R$,
$$
p_0(t_0,x,s,y)\geq\frac{1}{(2\pi\delta)^{d/2}}\exp\left\{-\frac{(R+|x|)^2}{2\delta}\right\}:=\frac{1}{K_\delta},
$$
where $p_0(t,x,s,y)$, $0\leq t\leq s$, $x\in\R, y\in\R$, is a transition probability density of a $d$-dimensional Wiener process.
For each $x\in\R$,
\begin{multline*}
\sup_{t_0\in[0,\infty)}\int_{t_0+\delta}^{t_0+T}ds\int_U\nu(s,dy)\leq K_\delta\sup_{t_0\in[0,\infty)}\int_{t_0+\delta}^{t_0+T}ds\int_{|y|\leq R} p_0(t_0,x,s,y)\nu(s,dy)\leq \\
K_\delta \sup_{t_0\in[0,\infty)}\sup_{\widetilde x\in\R}\int_{t_0}^{t_0+T}ds\int_{\R} p_0(t_0,\widetilde x,s,y)\nu(s,dy).
\end{multline*}
Fixed $C_0>0$, by the relation \eqref{Cond_A_prime} there exists $T_0>0$ such that
$$
\sup_{t_0\in[0,\infty)}\sup_{\widetilde x\in\R}\int_{t_0}^{t_0+T_0}ds\int_{\R} p_0(t_0,\widetilde x,s,y)\nu(s,dy)<C_0.
$$
Then  Remark \ref{Remark_triangle_ineq} implies that there exists $C_1>0$ such that
$$
\sup_{t_0\in[0,\infty)}\sup_{\widetilde x\in\R}\int_{t_0}^{t_0+T}ds\int_{\R} p_0(t_0,\widetilde x,s,y)\nu(s,dy)<C_1,
$$
which entails \eqref{eq_l234}. Now by Lemma \ref{Lemma_Converg_Densities} and \eqref{eq_l234} the last summand in the right-hand side of \eqref{eq_llll} tends to zero uniformly on $0\leq t_0<t_0+t\leq T$, $x\in\R$.

Thus we obtained that
$$
\sup_{0\leq t_0<t_0+t\leq T}\sup_{x\in\R} I_n^{1}(t_0,t,x)\to 0, \ \ n\to\infty.
$$
Using the similar argument we get
$$
\sup_{0\leq t_0<t_0+t\leq T}\sup_{x\in\R} I_n^{2}(t_0,t,x)\to 0, \ \ n\to\infty.
$$
This ends the proof.

\end{proof}

\end{document}